\documentclass[reqno,11pt]{amsart}  
\usepackage{amsmath, amssymb, amsthm, amsfonts, amsgen, stmaryrd}
\usepackage[english]{babel}
\usepackage{fullpage}
\usepackage[centertags]{amsmath}
\usepackage{indentfirst}
\usepackage{fancyhdr}
\usepackage[dvips]{graphicx}
\usepackage{psfrag}
\numberwithin{equation}{section}
\usepackage[latin1]{inputenc}

\title{The quantum Cartan algebra associated to a bicovariant differential calculus}
\author{Lucio S. Cirio, Chiara Pagani, Alessandro Zampini}

\address[]{\textit{Lucio Simone Cirio} \newline \indent Max Planck Institut f\"ur Mathematik, Vivatsgasse 7, 53111 Bonn, Germany \newline \indent \textit{Present address:} Grupo de Física Matemática
Complexo Interdisciplinar da Universidade de Lisboa, \newline \indent
Av. Prof. Gama Pinto, 2
1649-003 Lisboa,
Portugal}
\email{cirio@cii.fc.ul.pt }

\address[]{\textit{Chiara Pagani} \newline \indent  Institut f\"ur Analysis, Leibniz Universit\"at Hannover, Welfengarten 1, 30167 Hannover, Germany}
\email{pagani@math.uni-hannover.de}

\address[]{\textit{Alessandro Zampini} \newline \indent  Max Planck Institut f\"ur Mathematik, Vivatsgasse 7, 53111 Bonn, Germany}
\email{zampini@mpim-bonn.mpg.de}

\theoremstyle{plain}
\newtheorem{thm}{Theorem}[section]
\newtheorem{lem}[thm]{Lemma}
\newtheorem{prop}[thm]{Proposition}
\newtheorem{defi}[thm]{Definition}
\theoremstyle{remark}
\newtheorem{rem}[thm]{Remark}

\newcommand{\R}{{\mathbb{R}}}
\newcommand{\Z}{{\mathbb{Z}}}

\newcommand{\C}{\mathbb{C}}
\newcommand{\T}{\mathcal{T}}

\let\o=\omega
\let\a=\alpha
\let\b=\beta
\let\e=\eta
\newcommand{\la}{\triangleright}
\newcommand{\ra}{\triangleleft}
\newcommand{\co}{\triangle}

\newcommand{\sff}[2]{S(f_{#1#2})}
\newcommand{\sig}[4]{\sigma_{#1#2}^{#3#4}}
\newcommand{\dd}{{\rm d}}
\newcommand{\beq}{\begin{equation}}
\newcommand{\eeq}{\end{equation}}
\newcommand{\nn}{\nonumber}

\newcommand{\g}{\mathfrak{g}}
\newcommand{\gt}{\tilde{\g}}

\newcommand{\cart}{\mathcal{C}}

\newcommand{\even}{\mathcal{E}_{\Gamma}}

\newcommand{\h}{{H}}
\newcommand{\ext}{\Gamma^{\wedge}}

\newcommand{\lie}[1]{{L}_{#1}}
\newcommand{\liel}[1]{{L}_{#1}}
\newcommand{\lier}[1]{{L}^{\scriptscriptstyle{R}}_{#1}}
\newcommand{\intl}[1]{i_{#1}}
\newcommand{\intr}[1]{i^{\scriptscriptstyle{R}}_{#1}}

\newcommand{\lcartop}{(\liel{a}, \liel{jk}, \intl{a}, \dd)}

\newcommand{\de}{\mathfrak{\delta}}

\def\SUq{\mathrm{SU}_q(2)}
\def\usu{\mathfrak{U}_q(\mathfrak{su}(2))}
\newcommand{\ot}{\otimes}

\def\pa#1#2{\langle #1 , #2 \rangle}
\newcommand\op{\omega_+}
\newcommand\om{\omega_-}
\newcommand\oo{\omega_0}
\newcommand\oz{\omega_z}
\newcommand\x[1]{X_{#1}}
\newcommand\xp{X_+}
\newcommand\xm{X_-}
\newcommand\xo{X_0}
\newcommand\xz{X_z}
\newcommand\ii{\{-, +, z, 0 \}}
\def\sq{\mathrm{S}_q}

\newcommand{\hs}[2]{\left\langle #1,#2\right\rangle}  
\newcommand{\hp}[2]{\{ #1,#2\}}
\newcommand{\rinv}{\Gamma_{inv}}
\newcommand{\linv}{_{inv}\Gamma}
\newcommand{\dg}{\Delta_\Gamma}
\newcommand{\gd}{\;_\Gamma \Delta}
\newcommand{\rw}{\rightarrow}
\newcommand{\A}{A}
\newcommand{\M}{\mathcal{M}^\A_{bicov}}
\newcommand{\cartsu}{\cart_{\scriptscriptstyle{4D_+}}}

\newcommand{\gn}{\mathfrak{g}_0}

\newcommand{\gp}{\mathfrak{g}_1}
\newcommand{\gm}{\mathfrak{g}_{-1}}

\newcommand{\IZ}{\ensuremath{\mathbb{Z}}}

\begin{document}

\begin{abstract}
We associate to any (suitable) bicovariant differential calculus on a quantum group a Cartan Hopf algebra which has a left, respectively right, representation in terms of left,  respectively right, Cartan calculus operators. The example of the Hopf algebra associated to the
$4D_+$ differential calculus on $SU_q(2)$ is described.
\end{abstract}

\maketitle
\tableofcontents

\section{Introduction}
\label{s:cla}

This paper deals with the notion of Cartan calculus for quantum groups.
If a Lie group $G$ acts on a differentiable manifold $M$ there is an induced action by pullback of $G$ on the exterior algebra $\Omega(M)$ and in particular there is an induced action by Lie and inner derivations of the Lie algebra $\g$ of $G$ on $\Omega(M)$, giving to
$\Omega(M)$ the structure of
a \textit{$\gt$-module} \eqref{gt}. This notion, introduced by Cartan in \cite{cartan1} but appeared in the literature under various names, is at the base of the formulation of the
so-called Weil and Cartan models of equivariant cohomology (see e.g. \cite{gs}).
Our aim is to generalize to a noncommutative setting this algebraic approach to the classical Cartan calculus
(for the action of a Lie group G on itself) as a first step towards the construction of models of equivariant cohomology for quantum groups \cite{cpz}.

\bigskip
\noindent
We recall few facts about the classical theory before to move to the noncommutative setting.
Let $G$ be a  finite dimensional Lie group and $e_i$, $i=1, \dots , N=dim(G)$ be a basis for the Lie algebra $\g$ of $G$.
There is a $\IZ$-graded Lie algebra $\gt$ over $\C$ naturally associated with $\g$:
\begin{equation} \label{gt}
\gt= \bigoplus_{i \in \IZ} \g_i := \gm \oplus \gn \oplus \gp \; ,
\end{equation}
where $\gn$ and $\gm$ are $N$-dimensional vector spaces and $\gp$ is a $1$-dim vector space. Fixed  $e_i$,
$\xi_i$ and $\delta$ be basis for $\gn,\gm$ and $\gp$ respectively, the  $\IZ$-graded Lie algebra structure for $\gt$ is given by the bracket $[\, , \,]: \g_i  \times \g_j \rightarrow \g_{i+j}$ defined on the basis  elements as
\begin{equation} \label{laca}
\begin{array}{lllll}
[e_i, e_j]= c_{ij}^k e_k  ~  ,& \quad & [e_j, \xi_k]= c_{jk}^l \xi_l ~ ,& \quad& [\xi_k, \xi_l] = 0 ~ , \\
{[} \delta , \xi_k ] =e_k ~  , & &
[\delta, e_i] =0 ~  , & & [\delta,\delta]=0 ~ ,
\end{array}
\end{equation}
where $c_{ij}^k$ are the structure constants of $\g$ relative to the chosen basis.
We refer to $\gt$ as the \emph{`classical Cartan algebra'}.

As mentioned, if $G$ acts differentiably on a manifold $M$, then Lie and inner derivatives  act on $\Omega(M)$ giving a representation
of $\gt$ by graded derivations, see \cite{gs}.  This happens in particular for $G$ acting on itself via left and right multiplication $\mathrm{l}, \mathrm{r}:G\times G\to G$. For any $T\in\mathfrak{g}$ it is possible to introduce a vector field $R_{T}\in\mathfrak{X}(G)$ as the derivation on the algebra of functions $\mathcal{F}(G)$ on $G$, defined in terms of  the pull-back mapping $\mathrm{l}_{g}^{*}:\mathcal{F}(G)\to \mathcal{F}(G)$ induced by $\mathrm{l}_{g}$, the left multiplication by $g \in G$. On $f\in \mathcal{F}(G)$:
\beq
R_{T}(f)=\frac{\dd}{\dd s}\left. (\mathrm{l}_{\exp sT}^{*}(f))\right|_{s=0}.
\label{Rv}
\eeq
It is called the right invariant vector field associated to $T\in\mathfrak{g}$ since it is invariant under the push-forward mapping induced by the right multiplication.

Similarly for any $T\in\mathfrak{g}$ the vector field $L_{T}\in\mathfrak{X}(G)$ is defined as the infinitesimal generator of the  pull-back  mapping $\mathrm{r}^{*}_{g}$ induced by the right action $\mathrm{r}_{g}$:
\beq
L_{T}(f)=\frac{\dd}{\dd s}\left.(\mathrm{r}^{*}_{\exp sT}(f))\right|_{s=0}
\label{Lv}
\eeq
on any $f\in\,\mathcal{F}(G)$, and satisfies a property of left invariance under the push-forward of the left multiplication.

The left and right invariant vector fields $\{L_{i}\}, ~\{R_{i}\}$ associated to the basis elements $e_{i}\in\,\mathfrak{g}$,  form two $\mathcal{F}(G)$-bimodule basis of $\mathfrak{X}(G)$, and provide dual $\mathcal{F}(G)$-bimodule  basis of $\Omega^{1}(G)$ in terms of left and right invariant 1-forms $\{\omega_{i}\}, ~\{\eta_{i}\}$.
These are implicitly defined via the contraction operators as
$i_{L_{j}}\omega_{k}=\delta_{jk},\;
i_{R_{j}}\eta_{k}=\delta_{jk}$,
and satisfy a property of left ($l^{*}_{g}(\omega_{i})=\omega_{i}$) or of right ($r^{*}_{g}(\eta_{i})=\eta_{i}$) invariance. The action of the differential on 0-forms can be written as
$\dd f=\sum_i L_{i}(f)\,\omega_{i}= \sum_i R_{i}(f)\,\eta_{i}.$
From the $\mathcal{F}(G)$-bimodule structures of both $\mathfrak{X}(G)$ and $\Omega^{1}(G)$, it is immediate to compute $
R_{i}=J_{ik}L_{k}$, equivalently $\eta_{i}J_{ik}=\omega_{k}$,
with $J_{ik}\in\,\mathcal{F}(G)$ related to the adjoint representation of $G$ on $\mathfrak{g}$.

\noindent
The map $\lambda:\gt\to\mathrm{End}(\Omega(G))$ defined on generators by
\beq
e_{j}\mapsto\mathcal{L}_{L_j}, \qquad \xi_{j}\mapsto i_{L_j},\qquad \delta\mapsto\dd
\label{lacla}
\eeq
and extended as a $\Z_{2}$-graded Lie algebra homomorphism provides $\Omega(G)$ a left $\gt$-module structure by \emph{left} Cartan calculus derivations, as one can check by comparing relations \eqref{laca} with the Weil equations:
\begin{align}
[\mathcal{L}_{L_j},\mathcal{L}_{L_k}] & =c_{jk}^{h}\,\mathcal{L}_{L_h} \; , & [i_{L_j},i_{L_k}] & =0 \; , & [\mathcal{L}_{L_j},i_{L_k}] & =c_{jk}^{h}\,i_{L_h} \; , \nn \\
[\mathcal{L}_{L_j},\dd] & =0 \; , & [i_{L_j},\dd] & = \mathcal{L}_{L_j} \; , & [\dd,\dd] & 	=0 \; ,
\label{ccrleft}
\end{align}
for $j,k,h \in \{1, \dots , N \}$.
Analogously the map $\rho:\gt\to\mathrm{End}(\Omega(G))$ given on generators as
\beq
e_{j}\mapsto \mathcal{L}_{R_j}, \qquad \xi_{j}\mapsto i_{R_j},\qquad\de\mapsto - \dd
\label{racla}
\eeq
and extended as a $\Z_{2}$-graded Lie algebra anti-homomorphism provides $\Omega(G)$ with a right $\gt$-module structure in terms of the
\emph{right} Cartan calculus operators $\{\mathcal{L}_{R_j}, i_{R_j}, \dd\}$ satisfying
\begin{align}
[\mathcal{L}_{R_j},\mathcal{L}_{R_k}] & =-c_{jk}^{h}\,\mathcal{L}_{R_h} \, , & [i_{R_j},i_{R_k}] & =0 \, , & [\mathcal{L}_{R_j},i_{R_k}] & =-c_{jk}^{h}\,i_{R_h} \, , \nn \\
[\mathcal{L}_{R_j},\dd] & =0~, & [i_{R_j},\dd] & =\mathcal{L}_{R_j}~, & [\dd,\dd] & =0 \,.
\label{ccrright}
\end{align}
\bigskip

We hence move ahead to noncommutative geometry.
In \cite{wor,wor2} Woronowicz  developed a theory of bicovariant differential calculus on  quantum groups, introducing the notions of bicovariant bimodules, exterior algebra  and quantum tangent space $\T_\Gamma$ associated to a first order differential calculus $(\Gamma, \dd)$.
The problem of defining differential operators on quantum groups has been widely studied since then. Quantum Lie and inner derivatives appeared first in
\cite{ac} and \cite{swz1}, with the formulation of a differential geometry on  $GL_{q}(N), SL_{q}(N)$.
A general algebra of differential operators on a quasi-triangular Hopf algebra ${U}$ was given in \cite{sch4} in terms of a cross-product algebra ${U}\rtimes {H}$, with ${H}$ dually paired to ${U}$, thus generalising the classical semidirect product of two algebras using  the $R$-matrix structure of ${U}$. With ${A}$ and ${H}$ dually paired Hopf algebras, an algebra of graded derivative Cartan operators, whose action is given via the natural left action ${H}\la {A}$ coming from the pairing, has been studied in \cite{sch1, sch2, sch3, schef}. Lie and inner derivatives $\mathcal{L}_{h}, i_{h}$  giving a Cartan identity are at first introduced as acting on   the universal differential envelope $(\Omega({A}), \delta)$ -- that is on the universal calculus -- for any $h\in\, {H}$ so that    elements of ${H}$ are recovered as left-invariant vectors, and then to elements in ${A}\rtimes {H}$, which are intended as vector fields. The analysis is  extended further,  in order to consider graded derivatives acting on
 the exterior algebra $(\Omega( {A})_{\mathcal{T}_q}, \dd)$ coming from a calculus characterised by a suitable quantum tangent space $\mathcal{T}_{q}\subset\ker\,\varepsilon_{ {H}}$.  With $\chi\in\,\mathcal{T}_{q}$ one has degree 0 and (-1) operators $\mathcal{L}_{\chi}, i_{\chi}$ for left invariant vectors and for general vector fields, i.e. elements in $ {A}\rtimes \mathcal{T}_{q}$. This algebra of operators has a Hopf algebra structure. \\

 In this paper we first introduce  an analogue of the universal enveloping algebra of the classical Cartan algebra $\gt$, namely a quantum Cartan algebra $\cart_{\Gamma}$ associated to the differential calculus on a quantum group. 
Since our analysis is intended as a first step towards defining Weil and Cartan models for equivariant cohomology on quantum groups and their quantum homogeneous spaces,  we provide $\cart_{\Gamma}$  with both a  left and a  right representation in terms of graded differential operators  (Lie  and inner derivatives) ``dual" to the set of left (resp. right) invariant exterior forms in $\Gamma$, so to obtain a consistent left (right) action of $\cart_{\Gamma}$ on the exterior algebras defined on left (right) quantum homogeneous spaces.  
Different types of noncommutative algebras can be considered.
The case of algebras in which the deformation is generated by an abelian Drinfeld twist \cite{dr1,dr2} have been studied in \cite{lu} with the introduction of the corresponding quantum Cartan algebras and related twisted Weil and Cartan models for noncommutative equivariant cohomology.  In \cite{cpz} we aim to study  equivariant cohomology models for FRT bialgebras \cite{frt} and their dual class of Drinfeld-Jimbo algebras \cite{dr4, jim}, respectively quantum deformations of the coordinate algebra and universal enveloping algebra of classical Lie groups.  \\

The structure of the paper is the following. In section \ref{s:qca} we introduce a  $\Z_{2}$-graded bialgebra $\cart_{\Gamma}$, the \textit{quantum Cartan algebra}, associated to any  bicovariant first order differential calculus $(\Gamma, \dd)$ \`a la Woronowicz on a
quantum group. Under suitable conditions on the calculus, $\cart_{\Gamma}$ has a Hopf algebra structure.  In section \ref{s:qcc} we study the representations of $\cart_{\Gamma}$ in terms of left and right Lie and inner derivative operators acting on the external algebra $\Gamma^{\wedge}$ of the calculus. The last section contains the example of $\cartsu$, the quantum Cartan algebra associated to Woronowicz's $4D_{+}$ differential calculus on $SU_{q}(2)$.

\section{The quantum Cartan algebra}
\label{s:qca}

The aim of this section is to associate to any (suitable) first order differential calculus $(\Gamma, \dd)$  on a quantum group an Hopf algebra $\cart_\Gamma$, later referred to as the quantum Cartan algebra. We begin with a brief introductory subsection  to Woronowicz's theory.

\bigskip

\subsection{An outline of the theory of bicovariant differential calculi on  quantum groups}
\label{ss:rev}
We recall  few  basic aspects of the theory of bicovariant differential calculi on quantum groups. The aim is to fix the notations and introduce the objects used later in this paper; we refer the reader to the original paper \cite{wor} or  \cite[part IV]{ks}.
In what follows, summation over repeated indices is understood and we make use of Sweedler and Sweedler-like notations for coproduct and coactions.

Consider a unital Hopf algebra $(\A;\Delta,\varepsilon,S)$, with invertible antipode, endowed with a first order differential calculus $(\Gamma, \dd)$, where  $\dd:\A \rw \Gamma$ is a linear map which satisfies the Leibniz rule $\dd(xy)= (\dd x)y + x (\dd y)$, $x,y \in \A$ and $\Gamma$ is an $\A$-bimodule whose elements are of the form $\rho=x_k \dd y_k$, for $x_k,y_k \in \A$.

A first order differential calculus $(\Gamma, \dd)$ is left covariant  provided there exists a left $\A$-coaction  $\dg:\Gamma \rw \A \ot \Gamma$ such that  $\dg(x \dd y)= \Delta(x)(id \ot \dd) \Delta(y)$  for any
 $x, y \in \A$; analogously it is right covariant provided there exists a right $\A$-coaction $\gd: \Gamma \rw \Gamma \ot \A$ such that $\gd(x\dd y)=\Delta(x)(\dd\otimes id)\Delta(y)$ for any $x,y\in\,\A$. A first order differential calculus is then bicovariant provided it is both left and right covariant.
 Bicovariant differential calculi $(\Gamma, \dd)$ on ${A}$   are in one to one correspondence with right ideals $R_\Gamma\subset\ker\varepsilon$ which are invariant with respect to the (right) adjoint coaction of $\A$ on itself: $Ad(R_\Gamma)\subset R_\Gamma \ot \A $ where $Ad(x):= x_{(2)} \ot S(x_{(1)})x_{(3)}$ for any $x \in \A$ (\cite[Thm. 1.8]{wor} ). There is a left $\A$-module isomorphism $\Gamma\simeq\A\otimes(\ker\varepsilon/R_{\Gamma})$, and we refer to the dimension $N$ of the vector space $\ker\varepsilon/R_{\Gamma}$ as the dimension of the calculus, assumed to be finite in what follows. 

The triple $(\Gamma, \dg, \, \gd)$ is a bicovariant bimodule over $\A$.
 $\Gamma$ admits a free  $\A$-bimodule  basis of left-invariant elements $\omega_i \in \Gamma$ (that is
$\dg (\omega_i)= 1 \ot \omega_i$), with $i=1,\dots ,N$ and in analogy a free  $\A$-bimodule basis of right invariant elements $\eta_{i}\in\,\Gamma$ (i.e. such that $\gd(\eta_{i})=\eta_{i}\otimes 1$), $i=1,\ldots,N$.  Denote  $\linv\subset\Gamma$ (resp. $\rinv\subset\Gamma$) the subspace  of left (resp. right) invariant elements: one has that $dim(\linv)= dim (\rinv) = dim (\ker\varepsilon/R_{\Gamma})$.

The dual space $\A^{\prime}$ of functionals on $\A$ has natural left and right  actions on $\A$  given by $f \la x := x_{(1)} \cdot f (x_{(2)})$ and $x\ra f:=f(x_{(1)})\cdot x_{(2)}$ for any $f\in\,\A^{\prime},\,x\in\,\A$  induced by the pairing between $\A'$ (with the algebra structure $fg(x):=f(x_{(1)})g( x_{(2)})$) and the coalgebra $\A$  given by the evaluation of functionals
$\langle f, x \rangle := f(x)$.

The fundamental theorem for bicovariant bimodules \cite{wor} says that, considered an $N$-dimensional bicovariant bimodule
$\Gamma$ over $\A$ and
 given a basis $\o_{i}\in ~\linv\subset\Gamma$, there exist elements  $J_{ij}\in\,A$ and functionals
$f_{ij}\in\,\A^{\prime}$
, $i,j=1,\ldots,N$ such that for $x,y\in\,\A$:
\begin{enumerate}
\item  the $\A$-bimodule structure is given by
\begin{equation}
\label{fun_f}
\omega_i x= (f_{ij} \la x) \omega_j \, , \qquad   \qquad x \, \omega_i = \omega_j  ((f_{ij}\circ S^{-1}) \la x) ~ ;
\end{equation}
\item $\gd(\o_{i})=\o_j\otimes J_{ji}$~ ;
\item the $f_{ij}$ form an algebra representation of $\A$:
\beq
f_{ij}(xy)=f_{ik}(x)f_{kj}(y), \qquad f_{ij}(1)=\delta_{ij}
\label{fijxy}
\eeq
 where $1$ denotes the identity of $\A$;
\item $J=\{J_{ij}\}$ defines a matrix corepresentation of $\A$, i.e.:
\beq
\co(J_{ij})=J_{ik}\otimes J_{kj}, \qquad \varepsilon(J_{ij})=\delta_{ij} ~ ;
\label{coJ}
\eeq
\item the following  identity holds
\beq
J_{ki}(x\ra f_{kj})=(f_{ik}\la x)J_{jk} 
\label{coJf}
\eeq
\end{enumerate}
for $i,j,k= 1,\ldots,N$.

This theorem also proves an inverse of the above results, thus giving a characterisation of bicovariant bimodules (\cite{ks} \S 13.1).
 Moreover the theorem states that the elements $\e_j:= \o_i S(J_{ij})$ form a basis for $\rinv$. In terms of such a basis, the $\A$-bimodule structure of $\Gamma$ can be written as \cite{ac}:
\begin{equation}\label{fun_f_r}
\e_i x= (x \ra (f_{ij}\circ S^{-2})) \e_j \, ,  \qquad x \, \e_i = \e_j  (x \ra (f_{ij}\circ S^{-1})) 
\end{equation}
for any $x \in \A$.\\

\noindent\textsf{The quantum tangent space. \quad}
A central role in what follows will be played by the elements of the quantum tangent space $\T_\Gamma$ associated to the $N$-dimensional calculus $(\Gamma, \dd)$.
By definition, $\T_\Gamma$, simply denoted $\T$ in the following,  is the vector space
\begin{equation}
\T := \left\{ X \in \A' \,|\, X(1)=0 \mbox{ and } X(x)=0, \, \forall x \in R_\Gamma \right\}
\end{equation}
where $R_\Gamma\subset\ker\varepsilon$  characterizes the calculus. There exists a unique bilinear form
$\hp{~}{~}\; : \T\times\Gamma \rightarrow \C$ such that
\beq
\hp{X}{x\dd y} = \epsilon(x)\langle X,y \rangle  = \epsilon(x)X(y)
\label{dptg}
\eeq
 and with respect to this bilinear form the vector spaces $\T$ and $\linv$ form a  non-degenerate dual pairing, so that $dim(\T)= dim(\ker\varepsilon/R_{\Gamma})=N$. Being $N$
finite, the elements in $\T$ indeed belong to the dual Hopf algebra of $\A$ \cite[\S 1.2.8]{ks}.
Let us fix $\{X_i\}$ to be the  basis of $\T$ dual to the basis
$\{ \o_i \}$ of $\linv$ with respect to the pairing. It turns out that for any $x \in \A$,
the differential $\dd x$ can be written as 
\beq
\dd x= (X_i\la x)\, \o_i=-\e_i(x\ra S^{-1}(X_{i})) \, .
\label{ddx}
\eeq
The identity $\dd x=-\e_i(x\ra S^{-1}(X_{i}))$ can be proved using the formula
$(X_k \la x) J_{ik}= x \ra X_i$,  $x \in \A$ (see \cite[eq.(2.3.23)]{paolo}),
and the explicit
expression \eqref{antip-inv} of the inverse of the antipode: $S^{-1}(X_k)= - S^{-1}(f_{lk})  X_l$. Indeed applying the antipode,
 the former equation gives  $x \ra S^{-1}(X_j)= J_{ji} (S^{-1}(X_i) \la x)$, so that using \eqref{fun_f}, we get
 $$
 \begin{array}{ll}
 \dd x&=(X_i \la x) \o_i= X_i(x_{(2)}) \o_j \left((f_{ij}\circ S^{-1}) \la x_{(1)}\right)= (S^{-1}(f_{ij})X_i)(x_{(2)}) \o_j x_{(1)} \\
 &= -\o_j S^{-1} (X_j)(x_{(2)})  x_{(1)}= -\e_k J_{kj} (S^{-1}(X_j)\la x)=
 - \e_k (x \ra S^{-1}(X_k)) \; .
 \end{array}
 $$

The elements $X_i$, $R_{i}:=-S^{-1}(X_{i})$
 acting from the left and from the right on $\A$ satisfy
 the following quantum Leibniz rule
\begin{align}
&X_{i}\la(xy)=x (X_{i}\la y)+(X_{j}\la x)(f_{ji}\la y),\nn \\
&(xy)\ra R_{i}=(x\ra R_{i})y+(x\ra S^{-1}(f_{ji}))(y\ra R_{j}) \; .
\label{qLe}
\end{align}
\\

\noindent
\textsf{The braiding and higher order calculi. \quad} Bicovariant  bimodules over $\A$ form a braided category $\M$. Given two objects in the category, i.e. two bicovariant bimodules $\Gamma_1,\, \Gamma_2$ over $\A$, an element $\psi \in Mor(\Gamma_1,\Gamma_2)$ is a map $\psi: \Gamma_1 \rw \Gamma_2$ such that $\psi(a \rho b)=a \psi(\rho) b$ for all $a,\, b \in \A$ and $\rho \in \Gamma_1$ and such that $\psi$ intertwines the left and right coactions:
$$
(id \ot \psi)\circ \Delta_{\Gamma_1}= \Delta_{\Gamma_2} \circ \psi \, ,
\quad (\psi \ot id)\circ \, _{\Gamma_1}\Delta = \, _{\Gamma_2}\Delta \circ \psi  \, .
$$
The tensor product $\Gamma_1 \ot_\A \Gamma_2$  still belongs to $Obj(\M)$, as for  the tensor product  of a finite number of bicovariant bimodules. The bimodule structure and left and right coactions are introduced in a natural way, see e.g. \cite[\S13.1.4]{ks}.  In the following we omit the subscript $\A$ in $\ot_\A$ and write $\Gamma^{\ot 2}$ for $\Gamma \ot_\A \Gamma$.
For any bicovariant bimodule $\Gamma$,  there exists a unique morphims $\sigma \in Mor(\Gamma^{\ot 2}, \, \Gamma^{\ot 2})$ such that $\sigma(\omega \ot \eta)= \eta \ot \omega$ for any $\omega \in \linv$, $\eta \in \rinv$.
If $\o_i \in \linv$, $i=1, \dots ,N$ denotes a basis of  $\Gamma$, then
$\o_i \ot \o_j$, $i,j=1, \dots N$ form a basis of $\Gamma^{\ot 2}$.  On this basis, the braiding reads   $\sigma(\o_i\otimes\o_j)=\sigma^{kl}_{ij} \, \o_k\otimes\o_l$ with components $\sigma^{kl}_{ij}= f_{il}(J_{kj})$, while on the  right invariant forms $\e_j:= \o_i S(J_{ij})$ the identity
$\sigma(\e_i\otimes\e_j)=\sigma^{lk}_{ji} \, \e_k \otimes \e_l$ does hold (see \cite{wor}, \S 3).
\\

The braiding allows for the introduction of a quantum Lie algebra structure in the quantum tangent space.
For any $X, X^{\prime}$ in $\T$, define a linear functional on $\A$ by setting $[X,X^{\prime}](x)=\hs{X\otimes X^{\prime}}{\mathrm{Ad}(x)}$  with $x\in\,\A$. The property of bicovariance of the calculus ensures that the functional $[X,X^{\prime}]\in\,\T$, and that it satisfies a  Jacobi identity  $[X,[X^{\prime},X^{\prime\prime}]]=[[X,X^{\prime}],X^{\prime\prime}]-\sum_{n}[[X,T_{n}],T^{\prime}_{n}]$ where $\sigma^{t}(X^{\prime}\otimes X^{\prime\prime})=\sum_{n}T_{n}\otimes T^{\prime}_{n}$. In this expression, the map $\sigma^{t}$ on $\T\otimes\T$ is the transpose of the braiding, induced through the pairing between $\T$ and $\linv$ by imposing   $\hp{\theta\otimes\theta^{\prime}}{\sigma^{t}(X\otimes X^{\prime})}=\hp{\sigma(\theta\otimes\theta^{\prime})}{X\otimes X^{\prime}}$. In components $\sigma^t(\x{i}\ot \x{j})= (\sigma^t)^{kl}_{ij} ~\x{k} \ot \x{l}=  \sigma_{kl}^{ij}~\x{k} \ot \x{l} $.
For any element in $\T\otimes\T$ such that  $\sigma^{t}(\sum_{n}T_{n}\otimes T^{\prime}_{n})=\sum_{n}T_{n}\otimes T^{\prime}_{n}$, one has $\sum_{n}[T_{n},T^{\prime}_{n}]=0$, which can be read as an antisimmetry of the commutator.
On the basis $\{X_{i}\}$ in $\T$ this commutator acquires the form:

\beq
[X_{i},X_{j}]=X_{i}X_{j}-\sigma_{nk}^{ij}X_{n}X_{k}=\hs{X_{j}}{J_{im}}X_{m} \, .
\label{ja}
\eeq

\noindent 
Defined the $t_{ij}^{kl} \in \C $ such that $\ker(1-\sigma^t)$ is generated by elements $X_i\otimes X_j + t^{ij}_{kl} \, X_k\otimes X_l$ (with  $t^{ij}_{ij}=0$) it is
\begin{equation}
\label{t}
t^{ij}_{kl} - \sigma^{ij}_{kl} + \delta_{ik}\delta_{jl} - t^{ij}_{mn}\sigma^{mn}_{kl} = 0 \; .
\end{equation}
An explicit expression for $t$ can be found by direct computation once the calculus $\Gamma$ is chosen; later we shall present the Woronowicz's $4D_+$ calculus \cite{wor2} on $SU_q(2)$ as an example.
\\
The following identities hold
\begin{equation} \label{ff}
\sigma_{ij}^{rs} f_{rn} f_{sk} = f_{ir} f_{js} \sigma_{rs}^{nk}
\end{equation}
\begin{equation}\label{relt}
t_{ij}^{rs} f_{rn} f_{sk} = f_{ir} f_{js} t_{rs}^{nk} .
\end{equation}
Indeed the result is valid in more generality, being a condition that every morphism of $\Gamma^{\ot 2}$ has to satisfy as a compatibility condition with the $\A$-bimodule structure.
 Suppose $S$ in an automorphism of  $\linv^{\ot 2}$ with $S(\omega_{i}\ot \omega_{j})= S_{ij}^{kl} \omega_{k} \ot \omega_{l}$; we can then extend $S$ to $\Gamma\ot_{\A}\Gamma$ by $S(x \o_i\ot\o_j y): = x\, S(\o_i\ot \o_j) \, y$ for $x, \, y \in \A$. Then applying $S$ to both sides of the identity
$
\o_i \ot \o_j x = (f_{ir}f_{js} \la x) \o_r\ot \o_s
$
for $x \in \A$,  we get the result
$
S_{ij}^{rs} f_{rn} f_{sk} = f_{ir} f_{js} S_{rs}^{nk} \; .
$\\

In terms of the braiding $\sigma$ it is also possible to construct higher order differential calculi  from $(\Gamma,\dd)$.
The 'standard' exterior algebra defined by Woronowicz \cite{wor} is obtained as the quotient $\Gamma^{\wedge}_{\mathcal{\scriptscriptstyle{W}}}=\bigoplus_n (\Gamma^{\otimes^n} / \ker A_n)
$, where   $A_{n}:\Gamma^{\otimes^n}\to\Gamma^{\otimes^n}$ is a quantum antisymmetriser defined from the braiding.  A different  exterior  algebra can be introduced by considering the graded algebra defined as the quotient of the tensor algebra of $\Gamma$ by the two-sided ideal generated by $\ker(id-\sigma)$: $\ext:= \Gamma^{\otimes} / \langle \ker (id-\sigma) \rangle$
; one has $\Gamma^{\wedge}_{\mathcal{\scriptscriptstyle{W}}}\subset\Gamma^{\wedge}$ (see \cite[\S13.2.2]{ks}).  The differential is extended to $\Gamma^{\wedge}$  as the only degree one derivation such that $\dd^{2}=0$. Such an  exterior algebra $\Gamma^{\wedge}$ has  natural left and right $\A$-comodule structures, consistently given by recursively setting:
\beq
\dg(x\dd\alpha)=\Delta(x)(id\otimes\dd)\dg(\alpha),\qquad\qquad\gd(x\dd\alpha)=\Delta(x)(\dd\otimes id)\gd(\alpha)
\label{lrcoe}
\eeq
on any $\alpha\in\,\Gamma^{\wedge}$.  In our analysis, we shall consider this latter exterior algebra $\Gamma^{\wedge}$.

\subsection{The Hopf algebra $\cart_\Gamma$}
 
In this section we retain the notation  introduced before and consider $(\Gamma, \dd)$  an $N$-dimensional bicovariant calculus on a Hopf algebra $A$. We denote by $\h$ the dual Hopf algebra of $A$. In  the following, indices $i,j,\dots$ belong to $\{ 1,\dots, N\}$.  

 Fixed a free $\A$-bimodule
basis of left invariant elements $\o_{i}\in\,\linv\subset\Gamma$, the bimodule structure is uniquely determined  by  elements $f_{ij}\in\,\h$ and $J_{ij}\in\,A$ satisfying the conditions \eqref{fijxy}, \eqref{coJ}, \eqref{coJf}, while the elements $\omega_{i}$ select a quantum tangent space  with dual basis elements  $X_{i}\in\,\T\,\subset\h$. The relations expressed in \eqref{fijxy} together with \eqref{qLe} can be now
expressed in terms of the coalgebra and counity of $\h$ as
\begin{align}
\Delta(f_{ij})=f_{ik}\otimes f_{kj}\; ,\qquad\varepsilon(f_{ij})=\delta_{ij} \; , \nn \\
\Delta(X_{i})=1\otimes X_{i}\,+\,X_{j}\otimes f_{ji}\; , \qquad \varepsilon(X_{i})=0 \; .
\label{neco}
\end{align}
The subalgebra $\even\subset\h$ with generators $\{X_{i},f_{jk}\}$ associated to the bicovariant bimodule $\Gamma$  is characterised by the relations
\begin{align}
&X_i X_j-\sigma^{ij}_{kl}X_{k}X_{l}=C^{k}_{ij}X_{k} \; , \nn \\
&\sigma^{ij}_{nm}f_{ip}f_{jq}=f_{ni}f_{mj}\sigma^{pq}_{ij}\; , \nn \\
&C^{i}_{mn}f_{mj}f_{nk}+f_{ij}X_{k}=\sigma^{jk}_{pq}X_{p}f_{iq}+C^{l}_{jk}f_{il}\; ,\nn \\
&X_{k}f_{nl}=\sigma^{kl}_{ij}f_{ni}X_{j} \; ,
\label{quat}
\end{align}
where $C_{ij}^k:=\hs{X_{j}}{J_{ik}}=X_{j}(J_{ik})$ (see \cite{ac} for a proof of the third identity). Given
 \eqref{neco}, $\even$ turns out to be the sub-bialgebra in $\h$ encoding the structure of the bicovariant differential calculus
$(\Gamma, \dd)$ \cite{ber}.

\begin{rem}
The bialgebra structure of the quantum tangent space $\T$ associated to a bicovariant differential calculus on a Hopf algebra $A$ has been also studied in \cite{majid}. In that paper a different approach is followed; instead of considering the algebra $\even$, which contains also the elements $f_{ij}$, the authors look for a coproduct in the enveloping algebra of $\T$, thus working only with the generators $X_i$. They are able to define a \textit{braided} bialgebra structure on the enveloping algebra of the quantum tangent space, provided that the Hopf algebra $A$ is co-quasitriangular.
\end{rem}

In order to obtain a quantum version of the (universal enveloping algebra of the) classical Cartan algebra described in the previous section, we extend the algebra $\even$ giving the following

\begin{defi}
\label{defC}
The quantum Cartan algebra associated to an $N$-dimensional bicovariant differential calculus $\Gamma$ over $\A$ is the $\Z_2$-graded algebra $\cart_{\Gamma}$ generated over $\C$ by even elements $\{X_i,f_{jk}\}$ and odd elements $\{\xi_i,\de\}$ $(i,j,k=1,\ldots ,N)$, with relations \eqref{quat} among the even generators, and
\begin{equation}
\label{relcart}
\begin{array}{lllll}
 \xi_i\xi_j + t^{ij}_{kl}\, \xi_k \xi_l =0, &\qquad&   \xi_i \, \de + \de\, \xi_i = X_i, &\qquad& \de\de=0 \; ,\\
\xi_if_{kj} - \sigma^{ij}_{nm}f_{kn}\xi_m = 0, &\qquad&  \xi_iX_j - \sigma^{ij}_{kl}\, X_k\xi_l = C_{ij}^k\xi_k \; ,& &  \\
f_{ij}\de-\de f_{ij} = 0 , & \qquad &
X_i \, \de - \de \, X_i = 0. & &
\end{array}
\end{equation}
\end{defi}
\noindent Isomorphic calculi give rise to quantum Cartan algebras which are isomorphic.
\\
The next step is to introduce a bialgebra structure on $\cart_{\Gamma}$, extending that of  $\even$.
\begin{prop}
\label{biC}
The maps $\co:\cart_{\Gamma}\rightarrow\cart_{\Gamma}\otimes\cart_{\Gamma}$ and $\varepsilon: \cart_{\Gamma}\rightarrow\C$ defined on generators as
\begin{equation}
\label{coprC}
\begin{array}{lll}
\co(X_i) = 1\otimes X_i + X_j\otimes f_{ji}\; , & \qquad  & \co(f_{ij}) = f_{ik}\otimes f_{kj}\; , \\
\co(\xi_i) = 1\otimes \xi_i + \xi_j\otimes f_{ji}\; , & \qquad  & \co(\de) = 1\otimes \de + \de\otimes 1\; , \\
\varepsilon(X_i)= \varepsilon(\xi_i) = \varepsilon(\de) = 0\; , & \qquad  & \varepsilon(f_{ij})=\delta_{ij}
\end{array}
\end{equation}
and extended as $\Z_2$-graded algebra  homomorhisms define a bialgebra structure on $\cart_{\Gamma}$.
\end{prop}
\begin{proof}
It is evident that the identity  $(\varepsilon\otimes id)\co = id = (id\otimes\varepsilon)\co$ holds on generators of $\cart_{\Gamma}$; a further explicit computation, only using the identity $\co(f_{ij})=f_{ik}\otimes f_{kj}$, proves that the coproduct $\co$ is coassociative.  The nontrivial part of the statement is the compatibility of the coproduct with the relations \eqref{quat} and \eqref{relcart}.
We set
\begin{align}
&\mathfrak{a}_{ijk} := \xi_if_{jk} - \sigma^{ik}_{nm}f_{jn}\xi_m \; , \nn \\
& \mathfrak{b}_{ij} := \xi_i X_j - \sigma^{ij}_{kl}\, X_k\xi_l - C_{ij}^k\xi_k\; , \nn \\
&\mathfrak{d}_{ij} := \xi_i\xi_j + t^{ij}_{kl}\, \xi_k\xi_l \; , \nn \\
&\mathfrak{e}_{i}:=\xi_{i}\de+\de\xi_{i}-X_{i} \; , \nn \\
&\mathfrak{m}_{ij}:=f_{ij}\de-\de f_{ij} \; , \nn \\
&\mathfrak{z}_{i}:=X_{i}\de-\de X_{i} \; , \nn \\
&\mathfrak{t}:=\de^{2} \; .
\label{ABD}
\end{align}
Taking into account the $\Z_2$-grading on the tensor product algebra $\cart_{\Gamma}\otimes\cart_{\Gamma}$, so that
$$(a_1\otimes a_2)\cdot(b_1\otimes b_2)=(-1)^{|a_2||b_1|} a_1b_1\otimes a_2b_2 \qquad \forall \, a_i,b_i \, \in \cart$$
we compute directly the coproduct of $\mathfrak{a}_{ikj}$ at first:

\begin{equation*}
\begin{split}
\co(\mathfrak{a}_{ijk}) & =(1\otimes\xi_i + \xi_h \otimes f_{hi})(f_{jm}\otimes f_{mk} ) - \sigma^{ik}_{nm} (f_{jh}\otimes f_{hn})(1\otimes\xi_m + \xi_l \otimes f_{lm}) \\
 & = f_{jh}\otimes \xi_if_{hk} + \xi_h f_{jm}\otimes f_{hi}f_{mk} - \sigma^{ik}_{nm}f_{jh}\otimes f_{hn}\xi_m - \sigma^{ik}_{nm}f_{jh}\xi_l \otimes f_{hn}f_{lm} \\
 & = f_{jh}\otimes (\xi_i f_{hk} - \sigma^{ik}_{nm} f_{hn}\xi_n) + \xi_m f_{jn} \otimes f_{mi}f_{nk} - \sigma^{nm}_{hl} f_{jh}\xi_l \otimes f_{ni}f_{mk} \\
 & = f_{jh} \otimes \mathfrak{a}_{ihk} + (\xi_h f_{jl} - \sigma^{hl}_{mn}f_{jm}\xi_n) \otimes f_{hi}f_{lk} \\
 & = f_{jh}\otimes \mathfrak{a}_{ihk} + \mathfrak{a}_{mjn}\otimes f_{mi}f_{nk} \, .
\end{split}
\end{equation*}
This proves that $\mathfrak{a}_{ikj}=0$ implies $\co(\mathfrak{a}_{ikj})=0$, i.e.   the coproduct defined by \eqref{coprC} on generators of $\cart_{\Gamma}$ is  consistent with the algebraic relation $\mathfrak{a}_{ijk}=0$ in $\cart_{\Gamma}$.  Similarly:
\begin{equation*}
\begin{split}
\co(\mathfrak{b}_{ij}) & = (1\otimes\xi_i + \xi_m \otimes f_{mi})(1\otimes X_j+X_n\otimes f_{nj}) - \sigma^{ij}_{hp}(1\otimes X_h + X_r\otimes f_{rh})(1\otimes \xi_p + \xi_s\otimes f_{sp})  \\ & \phantom{=} \, - C_{ij}^k(1\otimes \xi_k+\xi_t\otimes f_{tk}) \\
 & = 1\otimes \xi_i X_j + X_n\otimes\xi_i f_{nj} + \xi_m\otimes f_{mi}X_j + \xi_m X_n\otimes f_{mi}f_{nj}  \\
 & \phantom{=} \, - \sigma^{ij}_{hl} (1\otimes X_h\xi_l + \xi_s\otimes X_h f_{sl} + X_r\otimes f_{rh}\xi_l + X_r\xi_s \otimes f_{rh}f_{sl} - C_{ij}^k(1\otimes \xi_k+\xi_t\otimes f_{tk}) \\
 & = 1\otimes \mathfrak{b}_{ij} + \xi_l \otimes (f_{li}X_j - \sigma^{ij}_{mn}X_m f_{ln} - C_{ij}^k f_{lk}) + \xi_m X_n \otimes f_{mi}f_{nj} - \sigma^{mn}_{rs} X_r\xi_s\otimes f_{mi}f_{nj} \\
 & = 1\otimes \mathfrak{b}_{ij} + \mathfrak{b}_{mn}\otimes f_{mi}f_{nj} + \xi_l \otimes (f_{li}X_j - \sigma^{ij}_{mn}X_m f_{ln} - C_{ij}^k f_{lk} + C_{mn}^l f_{mi}f_{nj}) \\
 & = 1\otimes \mathfrak{b}_{ij} + \mathfrak{b}_{mn}\otimes f_{mi}f_{nj}
\end{split}
\end{equation*}
where the last equality comes from the third relation in \eqref{quat}.
Using \eqref{relt} we can compute $\co(\mathfrak{d}_{ij})$:
\begin{equation*}
\begin{split}
\co(\mathfrak{d}_{ij}) & = (1\otimes\xi_i + \xi_r\otimes f_{ri})(1\otimes\xi_j + \xi_s\otimes f_{sj}) + t^{ij}_{kl} \,(1\otimes\xi_k +\xi_r\otimes f_{rk})(1\otimes\xi_l + \xi_s\otimes f_{sl}) \\
 & = 1\otimes\xi_i\xi_j - \xi_s\otimes\xi_i f_{sj} + \xi_r\otimes f_{ri}\xi_j + \xi_r\xi_s\otimes f_{ri}f_{sj}  \\
 & \phantom{=} \, + t^{ij}_{kl}\,( 1\otimes\xi_k\xi_l - \xi_s\otimes \xi_k f_{sl} + \xi_r\otimes f_{rk}\xi_l + \xi_r\xi_s\otimes f_{rk}f_{sl} ) \\
 & = 1\otimes \mathfrak{d}_{ij} + \xi_s\otimes (t^{ij}_{kl} - \sigma^{ij}_{kl} + \delta_{ik}\delta_{jl} - t^{ij}_{mn}\sigma^{mn}_{kl}) f_{sk}\xi_l + \xi_r\xi_s\otimes f_{ri}f_{sj} + t_{rs}^{kl}\xi_r\xi_s\otimes f_{ki}f_{lj} \\
 & = 1\otimes \mathfrak{d}_{ij} + \mathfrak{d}_{kl}\otimes f_{ki}f_{lj}
\end{split}
\end{equation*}
where we used \eqref{t} as well.
The structure of the proof is now clear, and along the same lines it is possible to prove:
\begin{align}
&\co(\mathfrak{e}_{i})=1\otimes \mathfrak{e}_{i}+\mathfrak{e}_{j}\otimes f_{ji}, \nn \\
&\co(\mathfrak{m}_{ij})=f_{ik}\otimes \mathfrak{m}_{kj}+\mathfrak{m}_{ik}\otimes f_{kj}, \nn \\
&\co(\mathfrak{z}_{i})=1\otimes \mathfrak{z}_{i}+\mathfrak{z}_{k}\otimes f_{ki}, \nn \\
&\co(\mathfrak{t})=0.
\label{altri}
\end{align}
Having defined $\cart_{\Gamma}$ as an extension of the coalgebra $\even$, the compatibility of the coproduct on even generators with  relations \eqref{quat} is already assured. We can nevertheless prove it as done for odd generators. We do not give the details of the computations but only list the relevant identities:

\beq
\begin{array}{ll}
\mathfrak{q}_{ij}=X_{i}X_{j}-\sigma^{ij}_{kl}X_{k}X_{l}-C^{k}_{ij}X_{k},  \qquad & \co(\mathfrak{q}_{ij})=1\otimes \mathfrak{b}_{ij}+\mathfrak{b}_{rs}\otimes f_{ri}f_{sj}; \\
\mathfrak{w}_{mnpq}=\sigma^{ij}_{mn}f_{ip}f_{jq}-f_{ni}f_{mj}\sigma_{ij}^{pq}, & \co(\mathfrak{w}_{mnpq})=\mathfrak{w}_{mnij}\otimes f_{ip}f_{jq}+f_{mi}f_{nj}\otimes\mathfrak{w}_{ijpq}; \\
\mathfrak{y}_{knl}=X_{k}f_{nl}-\sigma^{kl}_{ij}f_{ni}X_{j}, & \co(\mathfrak{y}_{knl})=f_{nm}\otimes\mathfrak{y}_{kml}+\mathfrak{y}_{mns}\otimes f_{mk}f_{sl}; \\
\mathfrak{r}_{ijk}=C^{i}_{mn}f_{mj}f_{nk}+f_{ij}X_{k}-\sigma^{jk}_{pq}X_{p}f_{iq}-C^{l}_{jk}f^{il}, & \co(\mathfrak{r}_{ijk})=f_{im}\otimes\mathfrak{r}_{mjk}+\mathfrak{r}_{imn}\otimes f_{mj}f_{nk} \, .
\end{array}
\label{conG}
\eeq
\end{proof}

Under suitable assumptions it is possible to define an antipode map $S:\cart_{\Gamma}\rightarrow\cart_{\Gamma}$ so that $\cart_{\Gamma}$ becomes a $\Z_2$-graded Hopf algebra. The Hopf algebra $\h$ has an antipode map that on generators of $\even$ satisfies 
\begin{align}
&S(f_{ij})f_{jk}=f_{ij}S(f_{jk})=\de_{ik}, \nn \\
&S(X_{i})+X_{j}S(f_{ji})=0 \, .
\label{idS}
\end{align}

When the antipode map $S:\h\to\h$ restricts to an antipode map in $\even$, that is $S(f_{ij}) \in \even$ for all $i,j \in \{1,\dots ,N\}$, we prove we can extend it to  a consistent $\Z_{2}$-graded antipode map in $\cart_{\Gamma}$.
 In the following we denote as '\textit{$S$-compatible}' a bicovariant first order differential calculus $(\Gamma,\dd)$ which has such a property. (It is not clear to us whether this assumption is satisfied by any bicovariant calculus or it is a non trivial request. Later we will consider the example of Woronowicz's $4D_+$ calculus on $SU_q(2)$ and show that such condition does hold in that case.) Before proving this result,  we list here some identities which will be later used.  These are obtained
after multiplying \eqref{quat}, \eqref{relcart} by the algebra elements $S(f_{ij})$ for suitable choice of indices, and using \eqref{idS}.

\begin{align}
\label{ids}
&\sig{i}{j}{r}{s} \sff{s}{n} \sff{r}{k} - \sff{j}{r} \sff{i}{s} \sig{s}{r}{k}{n}  = 0, \\
\label{ids2}
&\sff{p}{l}\xi_j - \sig{p}{n}{j}{k} \xi_n\sff{k}{l}  = 0, \\
\label{ids3}
&\sff{p}{q} e_i - \sigma_{pm}^{ij}e_m \sff{j}{q}  = 0, \\
\label{st}
&t_{ij}^{mn} \sff{p}{j}\sff{q}{i} - t_{qp}^{lk} \sff{k}{n} \sff{l}{m} =0,\\
\label{ids4}
&C_{mn}^i \sff{n}{k}\sff{m}{j} - e_r\sff{r}{k}\sff{i}{j} + \sig{p}{q}{j}{k}\sff{i}{q}e_r\sff{r}{p} + C_{jk}^l S(f_{il}) = 0.
\end{align}
\\

\begin{prop}
\label{thms}
Let $\cart_{\Gamma}$ be the quantum Cartan algebra associated to an $S$-compatible calculus $(\Gamma,\dd)$. The linear map $S_{\Gamma}:\cart_{\Gamma}\rightarrow\cart_{\Gamma}$ defined on generators in terms of the antipode in $\h$ by
\begin{equation}
\label{antip}
\begin{array}{lll}
S_{\Gamma}(f_{ij}) = S(f_{ij}), &\qquad & S_{\Gamma}(\xi_i) = -  \xi_j S(f_{ji}), \\
S_{\Gamma}(X_i)    = - \, X_j S(f_{ji}), &  \qquad & S_{\Gamma}(\de) = -  \de
\end{array}
\end{equation}
and extended as a $\Z_2$-graded algebra  and coalgebra anti-homomorphism defines an invertible antipode on the bialgebra $\cart_{\Gamma}$, so that $\cart_{\Gamma}$ becomes a $\Z_2$-graded Hopf algebra.
\end{prop}
For simplicity of notation in the following we will denote $S_{\Gamma}$ with  $S$.

\begin{proof} We start by showing that  the identity $\cdot (S\otimes id)\co (x) = \varepsilon (x) = \cdot (id\otimes S)\co(x)$ holds on any element $x\in\,\cart_{\Gamma}$ (where $\cdot$ denotes the algebra multiplication). On even generators $\{X_{i},f_{jk}\}\subset\even$, this is straightforward since it is valid in $\h$.  Using \eqref{idS} we compute
\begin{align*}
\cdot (S\otimes id)\co (\xi_i) &= \xi_i + S(\xi_j)f_{ji}
= \xi_i - \xi_k S(f_{kj})f_{ji} = 0 = \varepsilon (\xi_i).
\end{align*}
The analogous relation for $\de$ is trivial.

 Next we need to show that $S$ is a $\mathbb{Z}_2$-coalgebra anti-homomorpism on generators, i.e. that $\co\circ S = \tau\circ(S\otimes S)\circ \co$,  where $\tau$ is the flip map. On even generators it follows again from the property  of $S$ on $\h$; on $\de$ it is a direct application of the definition of $S$. We show the explicit computations on $\xi_a$:
\begin{equation*}
\begin{split}
\co\circ S(\xi_i) & = - \, \co (\xi_j S(f_{ji})) \\
& =- \, (1\otimes \xi_j + \xi_h\otimes f_{hj})(S(f_{ki})\otimes S(f_{jk})) \\
 & = - \, S(f_{ki})\otimes\xi_j S(f_{jk}) - \xi_h S(f_{ki})\otimes f_{hj}S(f_{jk}) \\
 &= S(f_{ki})\otimes s(\xi_k) + S(\xi_i)\otimes 1 \\
 & = \tau\circ(S\otimes S)(1\otimes\xi_i + \xi_k\otimes f_{ki}) = \tau\circ(S\otimes S)\circ\co(\xi_i) \, .
\end{split}
\end{equation*}

The further step is to prove that the map $S$ defined on the generators of $\cart_{\Gamma}$ is compatible with the algebraic relations \eqref{quat} and \eqref{relcart}. It is clear that relations \eqref{quat} involving only even generators  are compatible with $S$, since this compatibility follows from the Hopf algebra structure of $\h$.
We then focus our attention on the relations involving the odd degree generators.
Following the notation introduced in (\ref{ABD}), we begin by computing:
\begin{equation*}
\begin{split}
S(\mathfrak{a}_{ijk}) & = - \sff{j}{k}\xi_l\sff{l}{i} + \sig{n}{m}{i}{k} \xi_l\sff{l}{m}\sff{j}{n} \\
 & = - \, \sigma_{jr}^{ls}\xi_r\sff{s}{k}\sff{l}{i} + \sig{n}{m}{i}{k} \xi_l\sff{l}{m}\sff{j}{n} \\
 & = - \, \sigma_{sl}^{ik} \xi_r\sff{r}{l}\sff{j}{s} + \sig{n}{m}{i}{k} \xi_l\sff{l}{m}\sff{j}{n} = 0
\end{split}
\end{equation*}
where we used (\ref{ids2}) and (\ref{ids}) in the first and second line. Next we compute
\begin{equation*}
\begin{split}
S(\mathfrak{b}_{ij}) & = X_h\sff{h}{j}\xi_j\sff{j}{i} - \sig{l}{h}{i}{j}\xi_m\sff{m}{h}X_n\sff{n}{l} + C_{ij}^k \xi_t\sff{t}{k} \\
 & = \sig{h}{p}{l}{k} X_h\xi_p\sff{k}{j}\sff{l}{i} - \sig{h}{l}{i}{j}\xi_m\sff{m}{l}X_n\sff{n}{h} + C_{ij}^k \xi_t\sff{t}{k} \\
 & = \xi_h X_k\sff{k}{j}\sff{h}{i} - C_{hk}^t\xi_t\sff{k}{j}\sff{h}{i} - \sig{h}{l}{i}{j}\xi_m\sff{m}{l}X_n\sff{n}{h} + C_{ij}^k \xi_t\sff{t}{k} \\
 & = \xi_h \left( X_k\sff{k}{j}\sff{h}{i} - C_{tk}^h\sff{k}{j}\sff{t}{i} - \sig{l}{m}{i}{j}\sff{h}{m}X_n\sff{n}{l} + C_{ij}^k \sff{h}{k} \right) = 0
\end{split}
\end{equation*}
where this time we used (\ref{ids2}) in the first line, the relation between $\xi_j$ and $X_k$ from (\ref{relcart}) in the second line and finally (\ref{ids3}) in the last line. We go on with
\begin{equation*}
\begin{split}
S(\mathfrak{d}_{ij}) & = \xi_h\sff{h}{j}\xi_l\sff{l}{i} + t_{hl}^{ij} \xi_m\sff{m}{l}\xi_n\sff{n}{h} \\
 & = \sig{h}{r}{l}{s} \xi_h \xi_r\sff{s}{j}\sff{l}{i} + \sig{m}{p}{n}{t}t_{hl}^{ij}\xi_m\xi_p\sff{t}{l}\sff{n}{h} \\
 & = \sig{h}{r}{l}{s} \xi_h \xi_r\sff{s}{j}\sff{l}{i} + \sig{m}{p}{n}{t}t_{nt}^{hl}\xi_m\xi_p\sff{l}{j}\sff{h}{i} \\
 & = \xi_m\xi_p\sff{l}{j}\sff{h}{i} (\sig{m}{p}{h}{l} + t_{nt}^{hl}\sig{m}{p}{n}{t} ) \\
 & = \xi_m\xi_p\sff{l}{j}\sff{h}{i} (t_{mp}^{hl} + \delta_{hm}\delta_{lp}) \\
 & = - \, \xi_h\xi_l \sff{l}{j}\sff{h}{i} + \xi_h\xi_l \sff{l}{j}\sff{h}{i} = 0
\end{split}
\end{equation*}
where in this case we used (\ref{ids2}) on both terms of the first line, then (\ref{st}) on the second line, (\ref{t}) on the third line and the relation between $\xi_m$ and $\xi_p$ from (\ref{relcart}) in the fifth line. It is also
\begin{equation*}
\begin{split}
S(\mathfrak{e}_i) & = - \, \de \xi_j \sff{j}{i} - \xi_j\sff{j}{i} + X_j\sff{b}{i} \\
       & = - \, (\de\xi_j + \xi_j \de)\sff{j}{i} + X_j\sff{j}{i} = \mathfrak{e}_{j}S(f_{ji})
\end{split}
\end{equation*}
hence for $\mathfrak{e}_{a}=0$ it is $S(\mathfrak{e}_{a})=0$ as well. Along the same lines, it is easy to check that:
\begin{align}
&S(\mathfrak{t})=\mathfrak{t}, \nn \\
&S(\mathfrak{m}_{ij})=\de S(f_{ij})-S(f_{ij})\de=0, \nn \\
&S(\mathfrak{z}_{i})=-\de X_{j}S(f_{ji})+X_{j}S(f_{ji})\de=0.
\label{spu}
\end{align}
The last two identities in the above expressions are valid since the differential calculus $(\Gamma, \dd)$ is $S$-compatible, so that both $S(f_{ij})$ and $X_{j}S(f_{ji})$ are elements in the algebra $\even$, whose elements in $\cart_{\Gamma}$ commute with the element $\de$. It is then proved that the map $S$ introduced on the generators of $\cart_{\Gamma}$ by \eqref{antip} consistently define an antipode.

The invertibility of the antipode on even generators follows from the invertibility of $S$ in $\h$ and the hypothesis of $S$-compatibility of the calculus. Since the antipode has the same expression on the $X_i$ and on the $\xi_i$, the invertibility for odd generators also follows. Explicitly in terms of the elements   $S^{-1}(f_{ij}) \in \cart_\Gamma$ we can show that
\begin{equation}\label{antip-inv}
S^{-1}(X_i) = -  S^{-1}(f_{ji}) \, X_j  \; , \qquad S^{-1}(\xi_i) = -  S^{-1}(f_{ji}) \, \xi_j  \, .
\end{equation}
For instance on the $X_i$ one has
\begin{eqnarray*}
S^{-1}(S(X_i))&= & - S^{-1}(X_j S(f_{ji})) = -f_{ji} S^{-1}(X_j) = f_{ji} S^{-1}(f_{kj})X_k
= S^{-1}(S(f_{ji})) S^{-1}(f_{kj})X_k \\ &=& S^{-1}(f_{kj}S(f_{ji}) )X_k = \delta_{ki}~ X_k= X_i
\end{eqnarray*}
and also
$ S(S^{-1}(X_i)) = - S(X_j)f_{ji} = X_k S(f_{kj})f_{ji} = X_{k} ~\delta_{ki}= X_i $, where we used $S(f_{ik})f_{kj}=f_{ik} S(f_{kj})= \varepsilon(f_{ij})= \delta_{ij}$ and that $S^{-1}$ has to be an algebra anti-homomorphism.
\end{proof}

We conclude this section recalling that the classical differential calculus on a Lie group $G$ can be described as a bicovariant calculus \`a la Woronowicz, with $A=\mathcal{F}(G)$ and $R_\Gamma =(\ker\,\varepsilon)^{2}$. The associated tangent space reduces to $\mathfrak{g}$, with $\sigma^{ij}_{kl}= \delta_{ik}\delta_{jl}$, $f_{ij}= \delta_{ij}$ and $t^{ij}_{kl}= \delta_{ik}\delta_{jl}$, so that the `quantum' Cartan algebra from Definition \ref{defC} reduces to  the universal enveloping algebra of $\gt$.

\section{Quantum Cartan calculus}
\label{s:qcc}

As sketched in Section \ref{s:cla}, the classical Cartan algebra $\gt$ associated to a Lie group $G$ has two representations in terms of graded differential operators acting on $\Omega(G)$  associated to left and right invariant vector fields on $G$. This section describes analogous results for the quantum Cartan algebra $\cart_{\Gamma}$ of an $S$-compatible bicovariant differential calculus. The differential operators related to left invariant one forms were already defined (see e.g. \cite{sch4, ac, ks}); what we do is to exhibit their commutation relations by showing that they provide a left representation of the Hopf algebra $\cart_\Gamma$. We then define a right version of Lie and inner derivatives `dual' to the set of right invariant one forms and conclude they realise a right action of $\cart_{\Gamma}$.
\\

Given a bicovariant differential calculus $(\Gamma, \dd)$, the exterior algebra $\Gamma^{\wedge}$   with left and right compatible $\A$-coactions  given by \eqref{lrcoe} is a bicovariant bimodule. Dually it has a natural $\h$-bimodule structure:
\begin{align}
&h\la\alpha:=(id\otimes h)\circ\gd(\alpha)=\alpha_{(0)}\hs{h}{\alpha_{(1)}}, \nn \\
&\alpha\ra h:=(h\otimes id)\circ\dg(\alpha)=\hs{h}{\alpha_{(-1)}}\alpha_{(0)},
\label{deflie}
\end{align}
for any $h\in\,\h$ and $\alpha\in\,\Gamma^{\wedge}$.

\subsection{Left quantum Cartan calculus}
\label{ss:le}

Given the elements $X_{i}, ~ f_{ij}$ in $\even\subset\h$ we denote by $\liel{i}, \liel{ij}$ the following $0$-order operators on $\Gamma^\wedge$
\begin{equation}
\liel i(\alpha):=X_{i}\la \alpha \;, \quad
\liel {ij}(\alpha):=f_{ij}\la\alpha \; .
\label{dLL}
\end{equation}
The operator $\liel{i}:\Gamma^{\wedge}\to\Gamma^{\wedge}$ is usually referred to as the left Lie derivative associated to $X_{i}$. From the counit and coproduct maps in
\eqref{neco} in $\even$ the above operators act as derivations with $\liel{i}(1)=\liel{ij}(1)=0$ and the Leibniz rule
\begin{align}
&\liel{i} (\a\wedge\b) = \a\wedge \liel{i}(\b) + \liel{j}(\a)\wedge \liel{ij}(\b), \label{lielcopr}\\
&\liel{ij}(\a\wedge\b) = \liel{ik}(\a)\wedge \liel{kj}(\b) \label{fcopr}
\end{align}
on any $\alpha, \beta\in\,\Gamma^{\wedge}$. From \eqref{deflie} it is $\liel{i}(\alpha)= 0$ and $\liel{ij}(\alpha)=\delta_{ij} $ if $\gd(\alpha)= \alpha \otimes 1$; in addition
$\liel{i}( \o_k) = C_{ji}^{k}\o_j$ and $\liel{ij}(\o_k)= \sigma_{ik}^{lj}\o_l$. The relations \eqref{quat} can now be written as:
\begin{equation}
\begin{array}{ll}
\liel{i}\liel{j} -\sigma^{ij}_{kl}\liel{k}\liel{l} = C_{ij}^k\liel{k}, \qquad
&\sigma_{ij}^{rs} \liel{rn} \liel{sk} = \liel{ir} \liel{js} \sigma_{rs}^{nk} ,  \quad \\
C^{i}_{mn}\liel{mj}\liel{nk}+\liel{ij}\liel{k}=\sigma^{jk}_{pq}\liel{p}\liel{iq}+C^{l}_{jk}\liel{il}, \quad
&\liel{i} \liel{kj} = \sigma^{ij}_{nm} \liel{kn}\liel{m}\label{qual}.
\end{array}
\end{equation}

\noindent The differential $\dd:\ext\rightarrow\Gamma^{\wedge +1}$ is a degree $1$ derivation which commutes with the $\h$-bimodule structure maps,  so that
 \beq
 \liel{i}\dd - \dd\liel{i} = 0, \qquad\liel{ij}\dd - \dd\liel{ij} = 0.
\label{das}
\eeq
It is also possible to introduce degree $(-1)$ derivations on $\Gamma^{\wedge}$, which generalize the concept of inner product of a differential form by a vector field. Given the basis elements $\omega_{i}\in\,\linv\subset\Gamma$ the inner derivative associated to the dual basis element $X_{i}\in\T$  is defined as
\begin{equation}
\intl{k}(\o_{j}):=\hp{\o_{j}}{X_{k}}=\delta_{kj}
\label{icopr}
\end{equation}
and then extended to $\Gamma^\wedge$ by imposing the following `Leibniz rule'
\begin{equation}
\intl{k}(\a\wedge\b):= (-1)^{|\a|}\a\wedge \intl{k}(\b) + \intl{j}(\a)\wedge \liel{jk}(\b) \; , \quad \forall \alpha, \beta \in \Gamma^\wedge \,.
\end{equation}

\begin{rem}
The above relations have been proved \cite{ks} to consistently define  inner derivations on the exterior algebra $\Gamma^{\wedge}_{\scriptscriptstyle{\mathcal{W}}}=\oplus_{n}(\Gamma^{\otimes^{n}}/\ker\, A_{n})$ \`a la Woronowicz. Since $\langle\ker\,(id-\sigma)\rangle\,\subset\ker\,A_{n}$ for any $n=2,\ldots, N$, the definition is consistent also on $\Gamma^\wedge=\oplus_n \Gamma^{\otimes^n}/\langle\ker(id-\sigma)\rangle$, which is the exterior algebra we are considering in this paper.
\end{rem}
The $\Z_{2}$-graded commutation relations among inner derivatives and differential are still given by the Cartan's formula
\begin{equation}
\label{crid}
\intl{k}\dd + \dd\intl{k} = \liel{k} \; .
\end{equation}
The equation above is valid in this framework since Lie and inner derivatives are associated to elements in $\T$, whose counit is trivial.
As noted in \cite{sch3}, the Cartan identity needs to be modified if one is interested in defining $\lie{X}, ~ \intl{X}$ associated  to a generic element $X \in \h$. The condition that $\intl{X}$ vanishes on $0$-forms requires indeed in that case to add an extra terms $\varepsilon(X)$ to the Cartan identity.
\\

 We now prove that the operator $\lcartop$ indeed provides a representation of $\cart_{\Gamma}$ on $\ext$.

\begin{thm}
\label{leftcart}
Let $\cart_{\Gamma}$ be the quantum Cartan  algebra associated to an $S$-compatible calculus $(\Gamma,\dd)$. Then the map $\lambda:\cart_{\Gamma}\to \mathrm{End}(\Gamma^{\wedge})$, $x \mapsto \lambda_x$, defined on generators of $\cart_{\Gamma}$ by
\begin{equation}
\label{cartrep}
\lambda_{X_j}= \liel{j} \, , \qquad \lambda_{f_{jk}}=\liel{jk} \, , \qquad  \lambda_{\xi_j}=\intl{j}, \qquad \lambda_{\de}= \dd
\end{equation}
and extended as a $\Z_{2}$-graded algebra homomorphism to the whole of  $\cart_{\Gamma}$
defines a left $\cart_\Gamma$-Hopf module algebra structure on $\ext$.
\end{thm}
\begin{proof}

We have to prove that $\ext$ is a $\cart_{\Gamma}$-left module with respect to $\lambda$ and that for any generator $x$ of $\cart_{\Gamma}$
\begin{equation} \label{azione}
\lambda_x(\alpha\wedge\beta)=\lambda_{x_{(1)}}(\alpha) \wedge  \lambda_{x_{(2)}}(\beta)  \; , \quad
\lambda_x(1)=\varepsilon(x)1 \; .
\end{equation}

The statement involving the counit can be easily checked, while the one involving the coproduct follows by comparing Leibniz rules (\ref{lielcopr}), (\ref{fcopr}), (\ref{icopr}) and the one for $\dd$ with relations (\ref{coprC}).

Since the map $\lambda:\cart_{\Gamma}\to\mathrm{End}(\Gamma^{\wedge})$ is defined as a $\Z_{2}$-graded algebra homomorphism, the proof that $\ext$ is a left $\cart_\Gamma$-module reduces to show that the relations \eqref{quat} and \eqref{relcart} among generators are satisfied by their images via  $\lambda$.
The identities \eqref{qual} show that the relations \eqref{quat} involving only even degree generators in $\cart_{\Gamma}$ are satisfied.  With respect to the notation introduced  in \eqref{ABD}, the relation $\lambda_{\mathfrak{e}_{i}}=0$ is equivalent to the Cartan identity \eqref{crid}, while the relations $\lambda_{\mathfrak{m}_{ij}}=\lambda_{\mathfrak{z}_{i}}=0$ are equivalent to \eqref{das} and clearly $\lambda_{\mathfrak{t}}=0$ corresponds to the identity $\dd^{2}=0$.  It is enough to prove  that $\lambda_{\mathfrak{a}_{ijk}}=0$ and  $\lambda_{\mathfrak{b}_{ij}}=0$ on  $\linv$   and that $\quad\lambda_{\mathfrak{d}_{ij}}=0$ on $\linv^{\wedge 2}$,  since on a form of generic higher degree the result follows by induction using \eqref{azione} .
Explicitly one has:
\begin{align}
\lambda_{\mathfrak{a}_{hjk}}(\o_{s})&=(\intl{h}\liel{jk} - \sigma^{hk}_{nm}\liel{jn}\intl{m}) \, \o_s \nn\\
&= \intl{h} \sigma_{js}^{rk} \o_r - \sigma^{hk}_{nm}\liel{jn} \delta_{ms}
= \sigma_{js}^{rk} \delta_{hr} - \sigma^{hk}_{nm}\delta_{jn}\delta_{ms} = 0 ~,
\label{copri}
\end{align}
and also
\begin{align}
\lambda_{\mathfrak{b}_{hj}}(\o_{s})&=(\intl{h}\liel{j}-\sigma^{hj}_{kl}\liel{k}\intl{l}-C^{k}_{hj}\intl{k})\o_{s}=C^{s}_{lj}\intl{h}\o_{l}-C^{k}_{hj}\de_{ks}=0,
\label{coprii}
\end{align}
where we used that $\liel{a}(1)=0$, and finally
\begin{align}
\lambda_{\mathfrak{d}_{kj}}(\o_{s}\wedge\omega_{r})&=
(\intl{k}\intl{j} + t^{kj}_{hl}\, \intl{h}\intl{l}) \, \o_s\wedge\o_r \nn \\
&= \intl{k} (- \, \o_s\delta_{jr} + \delta_{rs} \liel{rj}\o_r ) + t^{kj}_{hl} \, \intl{h} ( -\, \o_s\delta_{lr} + \delta_{ms}\liel{ml}\o_r ) \nn\\
 & = \intl{k} (- \, \o_s\delta_{jr} + \sigma_{sr}^{mj}\o_m) + t^{kj}_{hl} \, \intl{h} ( -\, \o_s\delta_{lr} + \sigma_{sr}^{ml}\o_m) \nn\\
 & = - \delta_{ks}\delta_{jr} + \sigma_{sr}^{mj}\delta_{im} - t^{kj}_{hl}\delta_{hs}\delta_{lr} + t^{kj}_{hl}\sigma_{sr}^{ml}\delta_{hm}\nn \\
 & = - \delta_{ks}\delta_{jr} + \sigma_{sr}^{kj} - t^{kj}_{sr} + t^{kj}_{hl}\sigma_{sr}^{hl} = 0
\label{copriii}\end{align}
where in the last line we used (\ref{t}).
\end{proof}

Given a quantum subgroup $B$ of $A$, the bicovariant differential calculus $(\Gamma,\dd)$ induces  a right $A$-covariant differential calculus $\Gamma ' \subset \Gamma$ on the subalgebra $^{coB}A$ of left coinvariants in $\A$. The set of operators $\{\liel{i},\liel{ij},\intl{j},\dd\}$ restricts to a left representation of $\cart_{\Gamma}$ on the exterior algebra of $\Gamma'$.

\subsection{Right quantum Cartan calculus}
\label{ss:ri}

In this section we show how it is possible to define \textit{right} Lie and inner derivatives on $\ext$, so that together with the differential they provide a right representation of the quantum Cartan algebra $\cart_{\Gamma}$.
The notion of right Lie derivative on $\Gamma$ is based on the right $\h$-module structure of the calculus (\ref{deflie}): a fundamental role is played by the elements $R_i= - \, S^{-1}(X_i)$  in \eqref{ddx}.
\begin{defi} Given an element $X_{i}\in\,\T$, we define
the $0$-order operators:
\begin{align}
&\lier i:\Gamma^{\wedge}\to\Gamma^{\wedge}, \qquad\qquad\lier i(\alpha):=\alpha\ra R_{i}, \label{defrlie} \\
&\lier {ij}:\Gamma^{\wedge}\to\Gamma^{\wedge}, \qquad\qquad\lier {ij}(\alpha):=\alpha\ra S^{-1}(f_{ij}) \label{rlief}
\end{align}
and we refer to $\lier{i}$ as the right Lie derivative associated to  $X_{i}$.
\end{defi}
The right $\h$-module algebra structure of $\ext$ (dual to the left $\A$-comodule algebra structure) determines the following Leibniz rule for these right Lie derivatives:
\begin{equation}
\label{liercopr}
\begin{split}
\lier{i} (\a\wedge\b) & = \lier{i}(\a)\wedge\b + \lier{ji}(\a)\wedge \lier{j}(\b),  \\
\lier{ij}(\a\wedge\b) & = \lier{kj}(\a) \wedge \lier{ik}(\b)
\end{split}
\end{equation}
on any $\alpha, \beta\in\,\Gamma^{\wedge}$, as well as $\lier{a}(1)=\lier{ab}(1)=0$. We explicitly have
\begin{equation}
\label{lieract}
\begin{split}
\lier{i}(\e_k) & = - \, \e_k\ra(S^{-1}(X_i)) = - \, \langle S^{-1}(X_i), S(J_{lk}) \rangle \e_l = - \, C_{li}^k \e_l \\
\lier{ij} (\e_k) & = \e_k \ra (S^{-1}(f_{ij})) = \langle S^{-1}(f_{ij}) , S(J_{lk}) \rangle \e_l = \sigma_{ik}^{lj} \e_l
\end{split}
\end{equation}
while $\lier{i}\o_k = 0$ and $\lier{ij} \o_k = \delta_{ij}\o_k$.
From the properties of the antipode map in $\h$ it is easy to check that the algebraic relations \eqref{quat} defining $\even\subset\h$ give the following identities analogous to \eqref{qual}:
\begin{align}
&\lier{i}\lier{j} -\sigma^{ij}_{kl}\lier{k}\lier{l} = -C_{ij}^k\lier{k}\; , \qquad &\sigma_{ij}^{rs} \lier{rn} \lier{sk} = \lier{ir} \lier{js} \sigma_{rs}^{nk} \; ,\nn   \\
&C^{i}_{mn}\lier{mj}\lier{nk}-\lier{ij}\lier{k}=-\sigma^{jk}_{pq}\lier{p}\lier{iq}+C^{l}_{jk}\lier{il} \; ,
&\lier{i} \lier{kj} = \sigma^{ij}_{nm} \lier{kn}\lier{m}\label{quar}.
\end{align}
Since the differential $\dd:\ext\rightarrow\Gamma^{\wedge +1}$ is a degree $1$ derivation which commutes with the $\h$-bimodule structure one has, in analogy to \eqref{das},
 \beq
 \lier{i}\dd - \dd\lier{i} = 0 \; , \qquad\lier{ij}\dd - \dd\lier{ij} = 0 \; .
\label{dad}
\eeq
We need to define right inner derivatives. The idea is to consider $\intr{j}$ as the degree $(-1)$ derivation which on exact one-forms \eqref{ddx}  'extracts' the component along the right invariant element $\eta_j$, and to extend it by a Leibniz rule compatible with the coproduct of $R_j$.
In complete analogy with inner derivatives dual to left invariant forms, we give the definition on  $\Gamma^{\otimes}$ and then show that it descends to the exterior algebra $\ext$.
\begin{defi}
\label{defintr}
Given a  bicovariant differential calculus $(\Gamma, \dd)$ as before and $j \in \{1,\dots , N\}$, the right inner derivative $\intr{j}$ is defined as the degree (-1) operator $\intr{j}:\Gamma^{\otimes^k}\rightarrow \Gamma^{\otimes^{k-1}}$
satisfying
\begin{enumerate}
\item $\intr{j} (\e_i) := \delta_{ij} \, ,$ \hspace{1cm} $\forall \, i\in \{1,\dots , N\} $
\item $\intr{j} (\a\otimes\b) := \intr{j}(\a)\otimes\b + (-1)^{|\a|} \lier{kj}(\a)\otimes \intr{k}(\b)\, , \qquad\qquad \forall \, \a, \b \in \Gamma^{\otimes}. $
\end{enumerate}
\end{defi}
\begin{lem}
The right inner derivative $\intr{j}$ descends to a well defined operator on the exterior algebra $\ext=\Gamma^{\otimes}/\langle \ker\,(1-\sigma)\rangle$.
\end{lem}
\begin{proof}
We first prove that $\intr{j}$ is zero on $\ker(1-\sigma)\subset \Gamma\otimes\Gamma$. Define $s_{ji}^{lk}$ as the matrix of coefficients such that
$\e_i\otimes\e_j + s^{lk}_{ji}\e_k\otimes\e_l \in \ker(1-\sigma)$. Similarly to (\ref{t}), applying $(1-\sigma)$ to the previous expression we get the identity
\begin{equation*}
\delta_{im}\delta_{jn} + s_{ji}^{nm} - \sigma_{ji}^{nm} - s_{ji}^{lk}\sigma_{lk}^{nm} = 0,
\end{equation*}
that we use to compute
\begin{equation*}
\begin{split}
\intr{k} \left( \e_i\otimes\e_j + s^{nm}_{ji}\e_m\otimes\e_n \right) & = \intr{k}(\e_i)\otimes\e_j - \lier{rk}(\e_i)\otimes \intr{r}(\e_j)  \\
 & \phantom{=} \, + s^{nm}_{ji} \left( \intr{k}(\e_m)\otimes\e_n - \lier{rk}(\e_m)\otimes \intr{r}(\e_n) \right) \\
 & = \delta_{ki}\e_j - \sigma_{ri}^{nk}\e_n \delta_{rj} + s^{hm}_{ji} (\delta_{km}\e_h - \sigma_{rm}^{nk}\e_n\delta_{rh}) \\
 & = \e_n (\delta_{ki}\delta_{nj} - \sigma_{ji}^{nk} + s^{nk}_{ji} - s^{rc}_{ji}\sigma_{rc}^{nk} ) = 0 \, .
\end{split}
\end{equation*}
The fact that $\intr{j}$ is zero on the ideal in $\Gamma^{\otimes}$ generated by $\ker(1-\sigma)$ now follows by induction from the Leibniz rule.
\end{proof}

The following result is the analogue of Theorem \ref{leftcart}. We present it in terms of a more convenient  set of generators of $\cart_{\Gamma}$:

\begin{thm}
\label{thmcartr}
Let $\cart_{\Gamma}$ be a quantum Cartan calculus associated to an $S$-compatible calculus $(\Gamma,\dd)$. Then the map $\rho:\cart_{\Gamma}\to \mathrm{End}(\Gamma^{\wedge})$, $x \mapsto \rho_x$, defined on generators of $\cart_{\Gamma}$ by
\begin{equation}
\label{derho}
\rho_{\scriptstyle{R_j}} =  \lier{j}, \qquad \rho_{\scriptstyle{S^{-1}}(f_{jk})}=\lier{jk}, \qquad
\rho_{-\scriptstyle{S^{-1}}(\xi_j)} =\intr{j}, \qquad \rho_{\scriptstyle{S^{-1}}(\de)}= \dd
\end{equation}
and extended as a $\Z_{2}$-graded algebra anti-homomorphism to the whole of  $\cart_{\Gamma}$
defines a right $\cart_\Gamma$-Hopf module algebra structure on $\ext$.
\end{thm}
\begin{proof}
As in the Theorem \ref{leftcart}, we need to show that $\ext$ is a right $\cart_{\Gamma}$-module with respect to $\rho$ and that for any generator
$x$ of $\cart_\Gamma$
\begin{equation}\label{cond}
\rho_x(\alpha\wedge\beta)=\rho_{x_{(1)}}(\alpha) \wedge \rho_{x_{(2)}}(\beta) \; , \quad  \rho_x(1)= \varepsilon(x) \, .
\end{equation}
The identity $\rho_x(1)= \varepsilon(x)$  can again  be easily checked. By a straightforward direct calculation one has that in $\cart_{\Gamma}$
\begin{align}
&\co(S^{-1}(X_{i}))=S^{-1}(X_{i})\otimes 1+S^{-1}(f_{ji})\otimes S^{-1}(X_{j}), \nn \\
&\co(S^{-1}(f_{ij}))=S^{-1}(f_{jk})\otimes S^{-1}(f_{ki}), \nn \\
&\co(S^{-1}(\xi_{i}))=S^{-1}(\xi_{i})\otimes 1+S^{-1}(f_{ji})\otimes S^{-1}(\xi_{j}), \nn\\
&\co(S^{-1}(\de))=S^{-1}(\de)\otimes 1+1\otimes S^{-1}(\de), \label{coSx}
\end{align}
so that the other identity in \eqref{cond} comes from a comparison of these expressions above with the Leibniz rules \eqref{liercopr}, that for the differential and that for the right inner derivative given by $\intr{j} (\a\wedge\b) = \intr{j}(\a)\wedge\b + (-1)^{|\a|} \lier{kj}(\a)\wedge \intr{k}(\b)$ as from Definition \ref{defintr}.

In order to prove that $\ext$ is a right $\cart_{\Gamma}$-module
it is sufficient to show that the $\Z_{2}$-graded anti-homomorphism $\rho$ transforms the relations \eqref{relcart}  involving odd degree generators of $\cart_{\Gamma}$ into identities satisfied by the corresponding right operators (and $\dd$) acting on $\Gamma^{\wedge}$.  Moreover, it is enough to prove these results on the generators of $\Gamma_{inv}$ (or $\Gamma_{inv}\wedge\Gamma_{inv}$ if trivial on $\Gamma_{inv}$) since  by induction we can then conclude for the generic element in $\ext$ as well. From \eqref{ABD} we have

\begin{align}
\rho_{\scriptstyle{S^{-1}}(\mathfrak{a}_{hkj})}(\eta_{r})=(\intr{h}\lier{kj} - \sigma^{hj}_{nm}\lier{kn}\intr{m})\e_r
= \intr{h}\sigma_{kr}^{lj}\e_l - \sigma^{hj}_{nm}\lier{kn}\delta_{mr} = \sigma_{kr}^{hj} - \sigma^{hj}_{nr}\delta_{kn} = 0 \, , \label{non1}
\end{align}
and also the following
\begin{align}
\rho_{\scriptstyle{S^{-1}}(\mathfrak{b}_{ij})}(\eta_{r})&=-(\intr{i}\lier{j} - \sigma^{ij}_{kl}\lier{k}\intr{l} + C_{ij}^t \intr{t}) \e_r \nn\\
&=  \intr{i}C_{lj}^r\e_l + \sigma^{ij}_{kl}\lier{k}\delta_{lr} - C_{ij}^t\delta_{tr} \nn \\
 & =  C_{lj}^r\delta_{il} - C_{ij}^r = 0. \label{non2}
\end{align}
On $\Gamma_{inv}\wedge\Gamma_{inv}$, using \eqref{t}, we compute
\begin{align}
\rho_{\scriptstyle{S^{-1}}(\mathfrak{d}_{ij})}(\eta_{r}\wedge\eta_{s})&=(\intr{i}\intr{j} + t^{ij}_{mn} \, \intr{m}\intr{n})\e_r\wedge\e_s \nn \\ & = \intr{i}(\delta_{jr}\e_s - \lier{hj}\e_r\delta_{hs}) + t^{ij}_{mn} \intr{m}(\delta_{nr}\e_s - \lier{hn}\e_r\delta_{hs}) \nn \\
 & = \intr{i} (\delta_{jr}\e_s - \sigma_{sr}^{mj}\e_m) + t^{ij}_{mn}\intr{m} (\delta_{nr}\e_s - \sigma_{sr}^{md}\e_m) \nn \\
 & = \delta_{jr}\delta_{is} - \sigma_{sr}^{ij} + t^{ij}_{sr} - t^{ij}_{mn}\sigma_{sr}^{mn} = 0\, . \label{non3}
\end{align}
The relation involving the term $\mathfrak{e}_{i}$ gives the Cartan identity, which we prove separately on elements $y\in\A$ and exact 1-forms $\dd y\in\,\Gamma^{\wedge}$: 
\begin{equation*}
\begin{split}
\rho_{\scriptstyle{S^{-1}}(\mathfrak{e}_{j})} (y)&=(\dd\intr{j} + \intr{j} \dd - \lier{j}) y   = \intr{j} \dd y - \lier{j} y = \intr{j} \e_k (y \ra R_k) - y\ra R_j = 0 \, ; \\
\rho_{\scriptstyle{S^{-1}}(\mathfrak{e}_{j})} (\dd y)&=(\dd\intr{j} + \intr{j} \dd - \lier{j}) \dd y  = \dd\intr{j} (\e_k (y\ra R_k)) - \dd\lier{j}y = \dd(y\ra R_j) - \dd(y\ra R_j) = 0.
\end{split}
\end{equation*}
We are left with three relations: $\rho(\mathfrak{m}_{ij})=\rho(\mathfrak{z}_{i})=0$ is equivalent to \eqref{dad} and again $\rho(\mathfrak{t})=0$ corresponds to the identity $\dd^{2}=0$.
\end{proof}

With $B$ a quantum subgroup of $A$, the set of operators $\{\lier{i},\lier{ij},\intr{j},\dd\}$
restricts to a right representation of $\cart_{\Gamma}$ on the exterior algebra of the left-covariant differential calculus induced by $\Gamma$  on the subalgebra  of right coinvariants $A^{coB}\subset\A$.

Following a different perspective, inner and Lie derivatives are introduced in \cite{as,paolo} not only for elements in the quantum tangent space $\T$, but for any vector field $V\in\,\Xi$, with $\Xi$ an $A$-bimodule introduced from  $\T$. It is possible to see that there is no choice of an element $V\in\Xi$ such that the action of the  inner derivative $i_{V}$ associated to $V$ coincides with the action of an $\intr{j}$ 
(Definition \ref{defintr}) on the whole exterior algebra $\Gamma^{\wedge}$. In particular one can prove that even if it is possible to find a vector field $V_{j}\in\, \Xi$ such that $i_{V_{j}}(\eta_{k})=\delta_{jk}$, then the action of the  inner derivative $i_{V_j}$  does not agree with the action of $\intr{j}$ on higher order forms since the Leibniz rules they  satisfy cannot coincide.

\begin{rem} Theorems  \ref{leftcart} and  \ref{thmcartr} show that $\ext$ is a left and right $\cart_{\Gamma}$-Hopf module algebra. By the general theory on Hopf algebras we know that every left action can be turned into a right one by the antipode: on every $v\in V$, $V$ left $\h$-module algebra, we can define a right $\h$ action by $v\ra h:=S(h)\la v$. The right $\cart_{\Gamma}$-module structure of $\ext$ is not of this type; for example $\e_i\ra R_k \neq S(R_k)\la\e_i$, with the rhs always zero by right-coinvariance of $\e_i$.
\end{rem}

\section{An explicit example: the Hopf algebra $\cartsu$}
\label{example}

In this section we present the example of the quantum Cartan algebra $\cartsu$ associated to the $4D_+$ bicovariant differential calculus on $\SUq$ introduced by Woronowicz in \cite{wor2}.
This will show how the knowledge of the braiding matrix $\sigma$ and the calculus $\Gamma$ allows to explicitly derive the algebra and coalgebra structure of $\cart_\Gamma$.  Furthermore we can directly check the '\textit{$S$-compatibility}' of the calculus (see discussion at page \pageref{thms}).

We start by recalling the Hopf algebra structure of $\SUq$. As a $*$-algebra it is generated over $\C$ by two generators $a,c$ together with their $*$-conjugates
$a^*, c^*$ subject to commutation relations

\begin{equation}
\label{suq2}
\begin{array}{lll}
ac = qca  \; , \qquad & ac^*= q c^* a  \; ,\qquad  & cc^*=c^*c \; ,  \\
a a^* + q^2 c c^* =1  \; ,  & a^* a + c^* c=1 \; ,
\end{array}
\end{equation}
with  parameter of deformation $q\in\R\backslash \{0\} $.  For $q=1$ such algebra reduces to the commutative coordinate algebra on the group manifold $SU(2)$.
\noindent
The Hopf algebra co-structures $(\co, \varepsilon, S)$, compatible with the $*$-structure, are
\begin{equation}
\begin{array}{ll}
\co(a)= a \ot a -q c^* \ot c \; , \quad  \quad &
\co(c)= c \ot a + a^* \ot c \; , \quad
\\
\varepsilon(a)=1 \;, \quad  &  \varepsilon(c)=0 \; , \\
S(a)=a^* \; , \quad  \quad & S(c)=-q~c  \; ,
\end{array}
\end{equation}
For $q^4 \neq 0,1$, we consider the Drinfeld-Jimbo universal enveloping algebra $\usu$  generated by $E,F,K,K^{-1}$ with commutation relations 
\begin{equation}
\label{usl}
\begin{split}
KK^{-1} & = K^{-1}K =1 \;, \qquad KEK^{-1}=q ~E \; ,\qquad KFK^{-1}=q^{-1} F \; , \\
EF-FE & = \frac{K^2-K^{-2}}{q-q^{-1}} \; ,
\end{split}
\end{equation}
together with Hopf algebra structures
\begin{equation}
\label{usl-Hopf}
\begin{array}{lll}
\co(E)=E \ot K + K^{-1} \ot E \; , \quad & S(E)= - q E \; , & \varepsilon(E)= 0 \; ,
\\
\co(F)=F \ot K + K^{-1} \ot F   \; , & S(F)=-q^{-1}F \; , & \varepsilon(F)= 0 \; ,
\\
\co(K^{\pm 1})= K^{\pm 1}  \ot K^{\pm 1} \; , & S(K^{\pm 1})= K^{\mp 1} \; , & \varepsilon(K^{\pm 1})=1 
\end{array}
\end{equation}
and involution  $E^*=~F , ~~ F^*=  E ,~~ (K^{\pm 1})^*=K^{\pm 1} \,$ . The classical limit is obtained by $h\rightarrow0$ after  setting $q=e^{h}, ~K=e^{hH},~K^{-1}=e^{-hH}$. (We recall that a quantum deformation of the universal enveloping algebra of $su(2)$  can also be defined in a different not equivalent way, see e.g. \cite[\S 3.1.2]{ks}.)

The Hopf algebras $\SUq $ and $ \usu $ are dually paired: the non zero part of the pairing is given by  (see e.g \cite[\S 4.4.1]{ks})
\begin{equation}
\label{pairing}
\pa{K^{\pm 1}}{a} = q^{ \mp \frac{1}{2}} \;  , \qquad  \pa{K^{\pm 1}}{a^*}= q^{ \pm \frac{1}{2}} \; ,\qquad \pa{E}{c} =1 \; \; ,\qquad
\pa{F}{c^*}= - q^{-1} \; .
\end{equation}

We present the quantum tangent space $\T\subset\usu$ of Woronowicz's $4D_+$ first order bicovariant differential calculus $(\Gamma ,\dd)$  by the basis $\x{i},~ i \in \ii$ given by
\begin{equation}
\label{tangent}
\begin{array}{ll}
\xm & :=q^\frac{1}{2} FK^{-1} ~ ,\quad  \xz := \lambda^{-1} (K^{-2} -1) ~ , \quad \xp :=q^{-\frac{1}{2}} EK^{-1} \; ,\vspace{5pt}
\\
\xo &:=\lambda^{-2}[q(K^2 -1)+ q^{-1}(K^{-2}-1)] + FE \; ,
\end{array}
\end{equation}
where $\lambda= q-q^{-1}$.
The vector spaces $\T$ and $\Gamma$ form a nondegenerate dual  pairing with respect to the bilinear
form  defined by extending \eqref{pairing} as $\pa{X}{x\dd y}= \varepsilon(x) X(y)$.

We denote by $\{ \op ~ , \om ~ , \oz ~,\oo \}$ the dual elements forming a base for the space $_{inv}\Gamma$ of left-invariant one forms; by construction $\pa{\x{i}}{\o_j}= \delta_{ij}$ for $i,l \in \ii$. By definition $\dg(\omega_i)= 1 \ot \omega_i$, while the right coaction is written in terms of elements $J_{ki} \in \SUq$ for $k,i \in \ii$ by $\gd(\o_i)= \o_k \ot J_{ki}$. By direct computation, one easily determines the $J_{ki}$ as well as the $C_{ij}^k = \pa{X_j}{J_{ik}}$,  $i,j,k \in \ii$. First
\begin{equation*}
\begin{array}{llll}
J_{--}= (a^*)^2 \; , \quad& J_{-+}=-qc^2 \; , \quad & J_{-z}= (1+q^2)a^* c \; , \quad & J_{-0}= -q \lambda ~ a^* c  \; ,
\\
J_{+-}= -q(c^*)^2 \; , & J_{++}= a^2 \; ,& J_{+z}= (q+q^{-1})ac^* \; ,& J_{+0}=-\lambda ~ ac^*  \; ,
\\
J_{z-}= -q ~a^* c^* \; , & J_{z+}=-ac  \; , & J_{zz}=aa^* -cc^* \; , & J_{z0}= q \lambda ~cc^* \; ,
\\
J_{0-}= 0 \; , & J_{0+}=0 \; , & J_{oz}= 0 \; , & J_{00}= 1 \; ,
\end{array}
\end{equation*}
and then the nonzero structure constants are the following
\begin{equation*}
\begin{array}{ll}
C^-_{-0}= (q+q^{-1}) \; , \quad  C^-_{-z}= -C^-_{z-}= -q^{-1} \; , &
C^+_{+0}= (q+q^{-1}) \; , \quad  C^+_{+z}= -C^+_{z+}= q \; , \vspace{3pt}
\\
C^z_{z0}=C^z_{-+}= -C^z_{+-}=(q+q^{-1})\; , & 
C^0_{z0}= C^0_{-+}=-C^0_{+-}= - \lambda \; .
\end{array}
\end{equation*}

We also know that for any element $x \in \SUq$ we can write $\dd x= (\x{i} \la x) \o_i= x_{(1)} \x{i}(x_{(2)}) \o_i$; on generators we have
\begin{equation}
\begin{array}{ll}
\dd a = \frac{q}{q+1}~a ~\oz -q~c^* \op ~ +\xi a~ \oo \; , \quad &
\dd (a^*)=  c \om - \frac{1}{q+1}~a^* ~\oz  +\xi a^*~ \oo \; ,\vspace{2pt}
\\
\dd c = \frac{q}{q+1}~c ~\oz +a^* \op ~ +\xi c~ \oo \; ,  &
\dd (c^*)= -q^{-1} a \om - \frac{1}{q+1}~c^* ~\oz  +\xi c^*~ \oo \; ,
\end{array}
\end{equation}
where $\xi:= 1- \frac{q}{(q+1)^2}$. Inverse formulas are
\begin{equation}
\begin{array}{ll}
\om=c^* \dd a^* -q a^* \dd c^* \; , & \oz=a^* \dd a + c^* \dd c -a ~\dd a^* -q^2 c ~\dd c^* \; ,
\\
\op = a ~\dd c -q c ~\dd a \; , & \oo= \frac{q+1}{q^2 +q+1} (a^* \dd a + c^* \dd c + qa~\dd a^* + q^3 c ~\dd c^*) \; .
\end{array}
\end{equation}
It also holds $\oz^*=-\oz ~, ~ \op^*= - \om ~ , ~ \oo^* =-\oo$. In the classical limit $\oo$ goes to zero and the above reduces to the standard three-dimensional calculus on the classical two-sphere.

The $\SUq$-bimodule structure of $_{inv}\Gamma$ is expressed via elements $f_{ij} \in \usu$ such that $\o_i \,x = (f_{ij}\la x) \o_j$, for every $i, j \in \ii$. They are explicitly given by
\begin{equation}
\label{functionals_f}
\begin{array}{lll}
f_{--}= 1 \; , &   f_{-0}= q^{-\frac{1}{2}} KE  \; , &   f_{z-}= q^{\frac{1}{2}} \lambda FK^{-1} \; ,
\\
f_{zz}= K^{-2}  \; , &   f_{z+}= q^{-\frac{1}{2}} \lambda E K^{-1}  \; , &   f_{z0}= \lambda [FE + q^{-1} \lambda^{-2}(K^{-2} -K^2)] \; ,
\\
f_{++}= 1  \; , &   f_{+0}= q^{\frac{1}{2}} K F
 \; , &   f_{00}= K^2 
\end{array}
\end{equation}
and zero otherwise. Furthermore, together with the $J_{ki}\in\SUq$ introduced above, they allow for an explicit computation of the braiding $\sigma$. Along the $\{\o_i\}$ basis of $_{inv}\Gamma$ we write $\sigma(\omega_i \ot \omega_j):= \sigma_{ij}^{kl} \omega_k \ot \omega_l$, with $\sigma_{ij}^{nm}=\pa{f_{im}}{J_{nj}}$; with some work one can compute the non-vanishing components, which are 
\begin{equation}
\label{sig}
\begin{array}{llllllll}
&\sigma_{--}^{--}= 1  \; , &  \sigma_{zz}^{zz}= 1 \; , & \sigma_{zz}^{z0}= (q+q^{-1}) \lambda  \; , & \sigma_{zz}^{-+}=
(q+q^{-1}) \lambda \; , \vspace{5pt}
\\
&\sigma_{zz}^{+-}= - (q+q^{-1}) \lambda \; ,
&\sigma_{++}^{++}= 1 \; , &  \sigma_{00}^{00}=1  \; , & \sigma_{-z}^{z-}=1  \; ,\vspace{5pt}
\\
 & \sigma_{-z}^{-0}= 1+q^2  \; , &\sigma_{-+}^{+-}= 1  \; ,&\sigma_{z-}^{-0}= -(1 + q^{-2}) \; ,
&\sigma_{-+}^{z0}= -1 \; ,\vspace{5pt}
\\
&  \sigma_{-0}^{0-}= 1  \; ,
 & \sigma_{-0}^{-0}= -q \lambda  \; , & \sigma_{z-}^{z-}= q^{-1} \lambda  \; , &\sigma_{z-}^{-z}= q^{-2} \; ,\vspace{5pt}
\\
&\sigma_{z+}^{+z}= q^2 \; , &  \sigma_{z+}^{z+}= -q\lambda \; , & \sigma_{z+}^{+0}=1+q^2  \; , & \sigma_{z0}^{0z}=1  \; , \vspace{5pt}
\\
&\sigma_{z0}^{+-}= \lambda^2  \; ,
&\sigma_{z0}^{-+}= -\lambda^2 \; , &  \sigma_{z0}^{z0}= -\lambda^2  \; , & \sigma_{+-}^{-+}=1  \; , \vspace{5pt}
\\
& \sigma_{+-}^{z0}= 1 \; , &\sigma_{+z}^{z+}= 1  \; ,
&\sigma_{+z}^{+0}= -(1+q^{-2}) \; , &  \sigma_{+0}^{0+}= 1 \; , \vspace{5pt}
\\
& \sigma_{+0}^{+0}= q^{-1}\lambda \; , & \sigma_{0-}^{-0}= q^2 \; , &\sigma_{0z}^{z0}= 1 \; ,
&\sigma_{0+}^{+0}= q^{-2} \; .
\end{array}
\end{equation}

\noindent
All necessary ingredients to write down the commutation relations \eqref{quat} among even generators of $\cartsu$ are now provided.
The commutator \eqref{ja} is for instance given by (no sum on  $i$) :
\begin{equation}
\begin{array}{l}
\left[ X_i,X_i \right]= 0 \;  \;, \qquad i = -,+,z,0 \vspace{3pt}\\
\left[ \xo,X_i \right]= 0 \; , \qquad i = -,+,z \vspace{3pt}\\
\left[ \xm,\xp \right] =  -\left[\xp,\xm\right] = (q + q^{-1}) \xz - \lambda \xo \; ,\vspace{3pt}\\
\left[\xm,\xz\right]=  -\left[\xz,\xm\right] = - q^{-1} \xm \; ,\vspace{3pt}\\
\left[\xm,\xo\right]= (q + q^{-1}) \xm \; ,\vspace{3pt}\\
\left[\xp,\xz\right]=  -\left[\xz,\xp\right] = q \xp \; ,\vspace{3pt}\\
\left[\xp,\xo\right]= (q + q^{-1}) \xp \; ,\vspace{3pt}\\
\left[\xz,\xo \right] = (q + q^{-1}) \xz - \lambda \xo
 \; .
\end{array}
\end{equation}

\noindent
To complete the description of the algebra structure of $\cartsu$ and give the relations involving also the odd generators  $\xi_i$, $i \in \ii$,  we  need to describe the kernel of the operator $id-\sigma^t$.
\\
Two forms are defined to be the elements of the quotient $\Gamma^{\wedge 2}:=\Gamma^{\ot 2} / \sq$
where $\sq:=\ker(id-\sigma)$ and for this specific calculus it is computed to be
\begin{equation}
\sq=\left\{ \begin{array}{l}
\om \ot \om  \; , \quad \op \ot \op \; , \quad \oo \ot \oo  \; , \quad q^2\op \ot \oz  +\oz \ot \op  \; ,\vspace{5pt}
\\  \oz \ot \om  +q^{-2}\om \ot \oz \; ,
\quad \op \ot \oo  +\oo \ot \op  \; , \quad \om \ot \oo  +\oo \ot \om  \; , \vspace{5pt}
\\ \op \ot \om  +\om \ot \op  \; , \quad \oz \ot \oo  +\oo \ot \oz - \lambda^2 \om \ot \op \; , \vspace{5pt}
\\  \oz \ot \oz  - \lambda (q+q^{-1})\op \ot \om  \quad
\end{array}
\right\} ~.
\end{equation}
\noindent
We can here  compute explicitly the kernel of the operator $id-\sigma^t$:
\begin{lem} The kernel $T_q:=\ker(id-\sigma^t)$ is given by
\begin{equation} \label{kst}
\left\{ \begin{array}{l}
\xm \ot \xm  \; , \quad \xp \ot \xp  \; , \quad \xo \ot \xo  \; , \quad  \xz \ot \xz \; ,
\vspace{5pt} \\
\xp \ot \xm  +\xm \ot \xp  \; , \quad
\xp \ot \xz  +\xz \ot \xp  \; , \quad
\vspace{5pt}\\
\xm \ot \xz  +\xz \ot \xm
\; , \quad
\xo \ot \xz  +\xz \ot \xo - \xm \ot \xp  \; , \vspace{5pt}
\\ q^{-2} \xm \ot \xo  +  \xo \ot \xm - (1+q^{-2}) \xz \ot \xm \; , \vspace{5pt}
\\  q^2  \xp \ot \xo  +  \xo \ot \xp + (1+q^{2}) \xz \ot \xp \quad
\end{array}
\right\}~ .
\end{equation}
\end{lem}
\begin{proof}
One can prove by direct computation that the above elements are in $\ker(id -\sigma^t)$. We can conclude the proof  by comparison with the dimension of $\ker(id - \sigma)$.
\end{proof}

From \eqref{kst} we can derive the  form of the matrix $t^{ab}_{cd}$  for the $4D_+$ calculus  and hence the explicit form of equations \eqref{relcart} (involving even and odd generators) for the algebra $\cartsu$. For instance the commutation relations involving the odd generators $\xi_i$ read:
\begin{equation}
\label{ii}
\begin{array}{ll}
\xi_{+}  \xi_{+} =0 \; , &\xi_{-} \xi_{-}=0 \; , \\
\xi_{0} \xi_{0}=0  \; , & \xi_{z}\xi_{z}=0 \; ,\\
\xi_{0} \xi_{-} + q^{-2} \xi_{-} \xi_{0} -(1+q^{-2}) \xi_{z} \xi_{-} = 0 \; ,&
\xi_{0} \xi_{+} + q^{2} \xi_{+} \xi_{0} +(1+q^{2}) \xi_{z} \xi_{+} = 0 \; ,\\
\xi_{0} \xi_{z} + \xi_{z} \xi_{0}-\xi_{-} \xi_{+} =0 \; ,&  \xi_{+} \xi_{-}+\xi_{-} \xi_{+} =0 \; ,\\
\xi_{+} \xi_{z} + \xi_{z} \xi_{+}=0 \; , & \xi_{-} \xi_{z} + \xi_{z} \xi_{-} =0 \; .
\end{array}
\end{equation}

To complete the presentation of $\cartsu$ we are left with the explicit characterization of the antipode map $S$ of \eqref{antip}, coproduct and counit are indeed determined by the objects already computed.
In the discussion before Thm \ref{thms} we introduced the terminology of '\textit{$S$-compatible}' calculus; we can now verify that the $4D_+$ calculus on $\SUq$ has this property.

\begin{lem} The antipode of $\usu$ restricts to a well defined map in the sub-algebra generated by elements $f_{ab}\in\usu$ of the $4D_+$ calculus $\Gamma$. Then $\Gamma$ is '\textit{$S$-compatible}'.
\end{lem}
\begin{proof}
One can directly compute the image under the antipode of $\usu$ of the $f$'s from their explicit expression in \eqref{functionals_f}, and verify that indeed for any $f$, $S(f)$ is a polynomial expression in the $f$'s themselves. Alternatively, we can arrange the $f$'s of \eqref{functionals_f} in a matrix with indices running in the (ordered) set $\ii$:
\begin{equation}
\label{f_matrix}
f:=
\begin{pmatrix}
1 &0 & 0 & f_{-0}
\\
0 & 1 & 0 & f_{+0}
\\
f_{z-} & f_{z+} & f_{zz} & f_{z0}
\\
0 & 0& 0 & f_{00}
\end{pmatrix} ~ .
\end{equation}
Then we can introduce the $\usu$-valued matrix $S(f)$ defined as
\begin{equation}
\label{Sf_matrix}
S(f) :=
\begin{pmatrix}
1 &0 & 0 & -\, f_{-0}f_{zz} \\
0 & 1 & 0 & -\, f_{+0}f_{zz} \\
-\,f_{00}f_{z-} & -\,f_{00}f_{z+} & f_{00} & g \\
0 & 0& 0 & f_{zz}
\end{pmatrix} \vspace{3pt}
\end{equation}
where $g=f_{z-}f_{-0}+ f_{z+}f_{+0}- \, f_{z0}= \lambda[EF+ q^{-1} \lambda^{-2} (K^{2}-K^{-2})]$. It is now a direct computation to show that
$(S(f))_{\a\b}f_{\b\gamma} = \delta_{\a\gamma} = f_{\a\b}(S(f))_{\b\gamma}$, and indeed that $S(f_{\a\b}) = (S(f))_{\a\b} $.
\end{proof}

\noindent
We conclude by giving the inverse of the antipode on the $f$'s, since it is used in Theorem \ref{thmcartr}:
\begin{equation}
\label{Sf_inv}
S^{-1}(f) :=
\begin{pmatrix}
1 &0 & 0 & -\,f_{zz}f_{-0}  \\
0 & 1 & 0 & -\, f_{zz}f_{+0} \\
-\,f_{z-}f_{00} & -\,f_{z+}f_{00} & f_{00} & g \\
0 & 0& 0 & f_{zz}
\end{pmatrix} ~ . \vspace{3pt}
\end{equation}

\subsection*{Acknowledgments}
We thank Paolo Aschieri, Giovanni Landi and Yuri I.Manin for useful discussions. This work has been developed during a common visit of the authors at
Max-Planck-Institut f\"ur Mathematik in Bonn, we gratefully acknowledge their support and hospitality.  
We thank the Hausdorff Zentrum f\"ur Mathematik der Universit\"at Bonn and the Stiftelsen Blanceflor Boncompagni-Ludovisi (Stockholm) for their support.

\end{document}